\documentclass{article}

\usepackage{amsmath,amssymb,amsthm,bm}
\usepackage{bbm}
\usepackage{enumerate}
\usepackage{pifont}
\usepackage{setspace}
\usepackage{tikz-cd} 
\usepackage[utf8]{inputenc}

\newcommand{\ZB}{\mathbb{Z}}
\newcommand{\DB}{\mathbb{D}}

\newcommand{\RB}{\mathbb{R}} 
\newcommand{\TB}{\mathbb{T}} 
\newcommand{\CB}{\mathbb{C}}

\newcommand{\DC}{\mathcal{D}}

\newcommand{\HC}{\mathcal{H}}

\newcommand{\KC}{\mathcal{K}}
\newcommand{\LC}{\mathcal{L}}
\newcommand{\MC}{\mathcal{M}}

\newcommand*\conj[1]{\overline{#1}}
\newcommand*\clos[1]{\overline{#1}}

\renewcommand{\Re}{\operatorname{Re}}
\renewcommand{\Im}{\operatorname{Im}}

\makeatletter
\providecommand*{\diff}%
	{\@ifnextchar^{\DIfF}{\DIfF^{}}}
\def\DIfF^#1{%
	\mathop{\mathrm{\mathstrut d}}%
		\nolimits^{#1}\gobblespace}
\def\gobblespace{%
	\futurelet\diffarg\opspace}
\def\opspace{%
	\let\DiffSpace\!
	\ifx\diffarg(%
	\let\DiffSpace\relax
	\else
	\ifx\diffarg[%
	\let\DiffSpace\relax
	\else
	\ifx\diffarg\{%
	\let\DiffSpace\relax
	\fi\fi\fi\DiffSpace}
\makeatother

\theoremstyle{plain}
\newtheorem{theorem}{Theorem}[section]
\newtheorem*{theorem*}{Theorem \ref{thm:ModelTheorem}}

\newtheorem{lemma}[theorem]{Lemma}
\newtheorem{proposition}[theorem]{Proposition}

\theoremstyle{definition}

\newtheorem{XxmpX}[theorem]{Example}
\newenvironment{example}    
{\pushQED{\qed}\begin{XxmpX}}
	{\popQED\end{XxmpX}}
\newtheorem{question}[theorem]{Question}
\theoremstyle{remark}
\newtheorem{remark}[theorem]{Remark}

\title{Cyclic $m$-isometries, and Dirichlet type spaces}
\author{Eskil Rydhe\thanks{e.rydhe@leeds.ac.uk. School of Mathematics, University of Leeds, Leeds LS2 9JT, UK. Supported by the Knut and Alice Wallenberg foundation. The author acknowledges interesting conversations with Jonathan Partington, Eva Gallardo--Gutiérrez, and Alexandru Aleman, and also expresses his gratitude to the anonymous referee for carefully reading this manuscript.}}

\newcommand{\MU}{\vec{\mu}}
\newcommand{\NU}{\vec{\nu}}

\begin{document}

	\maketitle
	
	\begin{abstract}
		We consider cyclic $m$-isometries on a complex separable Hilbert space. Such operators are characterized in terms of shifts on abstract spaces of weighted Dirichlet type. Our results resemble those of Agler and Stankus, but our model spaces are described in terms of Dirichlet integrals rather than analytic Dirichlet operators. The chosen point of view allows us to construct a variety of examples. An interesting feature among all of these is that the corresponding model spaces are contained in a certain subspace of the Hardy space $H^2$, depending only on the order of the corresponding operator. We also demonstrate how our framework allows for the construction of unbounded $m$-isometries.
	\end{abstract}
	
	\section{Introduction}\label{sec:Introduction}
	
	Let $\HC$ denote a complex, separable Hilbert space, and $\LC$ the corresponding class of bounded linear transformations. Given a linear (possibly unbounded) Hilbert space operator $T$, we define the sesquilinear form
	\begin{equation*}
	(x,y)\mapsto \sum_{j=0}^m(-1)^{m-j}\binom{m}{j}\langle T^{j}x,T^j y\rangle_\HC,
	\end{equation*}
	where $x$ and $y$ belong to the domain of any power of $T$, and $m\in\ZB_{\ge 0}$. We say that $T$ is an \textit{$m$-isometry} if this form vanishes. If $T\in\LC$, then we define 
	\begin{equation}\label{eq:DefinitionBetaK}
	\beta_n(T)=\sum_{j=0}^n(-1)^{n-j}\binom{n}{j}T^{*j}T^{j},\quad n\in\ZB_{\ge 0}.
	\end{equation}
	Clearly $T\in\LC$ is an $m$-isometry if and only if $\beta_m(T)=0$. Even though some parts of this paper are relevant to unbounded operators, our primary concern is with bounded ones. Therefore, when we speak of $m$-isometric operators, we tacitly assume that these are bounded, unless this assumption is explicitly questioned or contradicted. The only time that we consider unbounded operators is in relation to Example \ref{ex:Unbounded2Isometry}.
	
	The study of $m$-isometries originates from the work of Agler \cite{Agler1990:ADisconjugacyTheoremForToeplitzOperators}. A brief introduction to $m$-isometries, as well as notation, and other concepts central to this note, is made in Section \ref{sec:Preliminaries}. For now, note that a $1$-isometry is just an ordinary Hilbert space isometry. 
	
	Let $\DC_a$ denote the space of functions which are analytic on the open unit disc $\DB$, and smooth on its closure $\clos{\DB}$. Define the operator $M_z:\DC_a\to\DC_a$ by $(M_zf)(z)=zf(z)$. By continuity, we may extend $M_z$ to a bounded linear operator $M_\lambda$ acting on the \textit{Hardy space} $\DC_\lambda^2=H^2$ (this choice of notation is explained in Remark \ref{remark:InterchangeableNotation}). The operator $M_\lambda$ is an isometry, $\dim\ker M_\lambda^*=1$, and $\bigcap_{n\in\ZB_{\ge 0}}M_\lambda^nH^2=\{0\}$. It is a classical result that these properties determine $M_\lambda$ up to unitary transformations:
	
	\begin{proposition}\label{prop:1Isometries}
		Let $T\in\LC$ be an isometry, such that $\dim\ker T^*=1$, and $\bigcap_{n\in\ZB_{\ge 0}}T^n\HC=\{0\}$. Then there exists a unitary map $U:\HC\to H^2$ such that $T=U^*M_{\lambda}U$.
	\end{proposition}
	
	An operator $T\in\LC$ satisfying $\bigcap_{n\in\ZB_{\ge 0}}T^n\HC=\{0\}$ is often called \textit{analytic}. A more general statement than Proposition \ref{prop:1Isometries} is that an analytic isometry $T$ is determined by $\dim\ker T^*$. This was observed in \cite{Halmos1961:ShiftsOnHilbertSpaces}. A systematic treatment is given in \cite[Chapter 1]{Rosenblum-Rovnyak1985:HardyClassesAndOperatorTheory}. 
	
	A motivation for choosing the word ``analytic'' is that, if $M_X$ denotes multiplication by $z$, defined on a space $X$ of analytic functions on $\DB$, and if $f\in\bigcap_{n\in\ZB_{\ge 0}}M_X^nX$, then $f$ vanishes identically. This means that, if we want $T\in\LC$ to resemble $M_X$ in a reasonable sense, then it is natural to assume that $T$ is analytic. 
	
	Given a finite positive (regular Borel) measure $\mu$ on $\TB$, we let $M_{\lambda,\mu}$ denote the extension of $M_z$ to the \textit{$\mu$-weighted Dirichlet space $\DC_{\lambda,\mu}^2$}, i.e. the space of analytic functions $f:\DB\to\CB$ for which
	\[
	\|f\|_{\lambda,\mu}^2:=\lim_{r\to 1^-}\frac{1}{2\pi r}\int_{r\TB}|f(\zeta)|^2 \diff \lambda(\zeta)+\frac{1}{\pi}\int_{\DB}|f'(z)|^2 P_\mu(z)\diff A(z)<\infty.
	\]
	In the above expression, $\diff \lambda$ and $\diff A$ respectively signify integration with respect to arc length measure on the unit circle $\TB$, and area measure on $\CB$. $P_\mu$ denotes the Poisson extension of $\mu$ to $\DB$.	
	
	The following generalization of Proposition \ref{prop:1Isometries} is due to Richter
	\cite[Theorems 5.1 and 5.2]{Richter1991:ARepresentationTheoremForCyclicAnalyticTwo-Isometries}: 	
	
	\begin{proposition}\label{prop:2Isometries}
		Let $\mu$ be a positive finite measure on $\TB$. Then $M_{\lambda,\mu}$ is a bounded analytic $2$-isometry, with $\dim\ker M_{\lambda,\mu}^*=1$. Conversely, if $T\in\LC$ is an analytic $2$-isometry, with $\dim\ker T^*=1$, then there exists a finite positive measure $\mu$ on $\TB$ and a unitary map $U:\HC\to \DC_{\lambda,\mu}^2$ such that $T=U^*M_{\lambda,\mu}U$. Moreover, if $T_1$ and $T_2$ are associated with the respective measures $\mu_1$ and $\mu_2$, then $\mu_1=\mu_2$ if and only if $T_1$ and $T_2$ are unitarily equivalent.
	\end{proposition}
	
	The above result is a so-called \textit{model theorem}: $M_{\lambda,\mu}$ acting on the \textit{model space} $\DC_{\lambda,\mu}^2$ is an \textit{operator model} for $T$. Among model theorems, one should distinguish the so-called \textit{universal} model theorems, where every $T$ of a certain class is modelled by the same operator $S$ restricted to an $S$-invariant subspace (in general) depending on $T$, e.g. \cite[Section 1.5]{Rosenblum-Rovnyak1985:HardyClassesAndOperatorTheory}. In this sense, Proposition \ref{prop:2Isometries} is not a universal model theorem for the class of analytic $2$-isometries with $\dim\ker T^*=1$, but rather describes an individual model for each such operator.
	
	Let $\bigvee S$ denote the closed linear hull of $S\subset\HC$. An operator $T$ is called \textit{cyclic} if there exists a vector $e\in\HC$ such that $\bigvee\{T^ne;n\in\ZB_{\ge 0}\}=\HC$. The vector $e$ is also called \textit{cyclic}. Another result by Richter \cite[Theorem 1]{Richter1988:InvariantSubspacesOfTheDirichletShift} states that if $T$ is an analytic $2$-isometry, then $T$ has the \textit{wandering subspace property}, i.e. $\bigvee\{T^n\ker T^*;n\in\ZB_{\ge 0}\}=\HC$. In particular, if $T$ is an analytic $2$-isometry, and $\dim\ker T^*=1$, then any non-zero vector $e\in\ker T^*$ is cyclic for $T$. Proposition \ref{prop:2Isometries} also has a natural analogue where the condition $\dim\ker T^*=1$ is omitted, see \cite{Olofsson2004:AVonNeumann-WoldDecompositionOfTwo-Isometries}.
	
	The wandering subspace property for higher order isometries has been studied by Shimorin \cite{Shimorin2001:Wold-TypeDecompositionsAndWanderingSubspacesForOperatorsCloseToIsometries}. However, it remains unknown whether or not an arbitrary analytic $m$-isometry, with $m\ge 3$, necessarily has the wandering subspace property. For this reason, we will henceforth replace the corresponding hypothesis in Proposition \ref{prop:2Isometries}, that $T$ is analytic and $\dim\ker T^*=1$, with the assumption that $e$ is a cyclic unit vector for $T$. The normalization $\|e\|_\HC^2=1$ is by no means essential, but is added for convenience. We do not insist that $e\in\ker T^*$.
	
	Let $\DC'$ denote the space of distributions on $\TB$. Given $\mu\in\DC'$ and $f\in\DC_a$, we define the corresponding \textit{weighted Dirichlet integral} of order $n\in\ZB_{\ge 1}$ by
	\begin{equation}
	\DC_{\mu,n}(f):=\lim_{r\to 1^{-}}\frac{1}{n!(n-1)!\pi }\int_{r\DB}|f^{(n)}(z)|^2P_{\mu}(z)(1-|z|^2)^{n-1}\diff A(z).
	\end{equation}
	Here $f^{(n)}=\frac{\diff^nf}{\diff z^n}$. For $n=0$ we define
	\begin{equation}
	\DC_{\mu,0}(f) := \lim_{r\to 1^{-}} \frac{1}{2\pi r}\int_{r\TB}|f(\zeta)|^2P_\mu(\zeta)\diff \lambda (\zeta)=\mu(|f|^2).
	\end{equation}
	The last equality follows from Lemma \ref{lemma:DirichletIntegralAsADoubleSum}, which also implies that $\DC_{ \mu,n}(f)$ is well-defined. It will be convenient to adopt the convention that $\DC_{ \mu,n}(f)=0$ whenever $n<0$.
	
	Let $\MU=(\mu_0,\ldots,\mu_{m-1})\in\left(\DC'\right)^m$, and define the quadratic form
	\begin{equation}\label{eq:DefWeightedDirichletNorm}
	\|f\|_{\MU}^2=\sum_{n=0}^{m-1}\DC_{\mu_n,n}(f),\quad f\in\DC_a.
	\end{equation}
	We say that the tuple $\MU$ is \textit{allowable} if there exists $C>0$ such that 
	\begin{equation}\label{eq:AllowableSequence}
	0\le \|M_zf\|_{\MU}^2\le C^2\|f\|_{\MU}^2,\quad f\in\DC_a,
	\end{equation}
	i.e. $M_z$ is a bounded operator with respect to the positive semi-definite form $\|\cdot\|_{\MU}^2$. Since $\|1\|_{\MU}^2=\hat \mu_0(0)$, \eqref{eq:AllowableSequence} implies that $\|\cdot\|_{\MU}^2\equiv 0$ if and only if $\hat \mu_0(0)=0$. We call the allowable $m$-tuple \textit{normalized} if $\hat \mu_0(0)=1$. 
	
	Given an allowable $m$-tuple $\MU$, we let $\KC_{\MU}=\ker\|\cdot\|_{\MU}$. Define $\DC_{\MU}^2$ as the completion of $\DC_a/\KC_{\MU}$ with respect to the norm $\|\cdot\|_{\MU}$. By \eqref{eq:AllowableSequence}, $M_z$ is well defined on $\DC_a/\KC_{\MU}$, and may be uniquely extended to a bounded linear operator $M_{\MU}:\DC_{\MU}^2\to\DC_{\MU}^2$. 
	
	Our main result is stated and proved in Section \ref{sec:ModelTheorem}. We restate it here for convenience:
	
	\begin{theorem*}
		If $\MU\in\left(\DC'\right)^m$ is a normalized allowable $m$-tuple, then the operator $M_{\MU}$ is a bounded $m$-isometry, with $1$ as a cyclic unit vector.	Conversely, if $T\in\LC$ is an $m$-isometry with a cyclic unit vector $e$, then $\MU\in\left(\DC'\right)^m$ given by
		\[
		\hat \mu_n(k)=\conj{\hat \mu_n(-k)}=\langle \beta_{n}(T)e,T^ke\rangle_\HC,\quad k\in\ZB_{\ge 0},
		\]
		is a normalized allowable $m$-tuple, and there exists a unitary map $U:\HC\to \DC_{\MU}^2$ such that $T=U^*M_{\MU}U$, and $Ue=1$. 
		
		If $T_j:\HC_j\to\HC_j$, $j\in\{1,2\}$, are bounded $m$-isometries with cyclic vectors $e_j$, then the associated $m$-tuples $\MU_{j}$ coincide if and only if there exists a unitary map $U:\HC_1\to\HC_2$ such that $T_1=U^*T_2U$ and $Ue_1=e_2$.
	\end{theorem*}
	
	\begin{remark}
		Note that for $\MU_1$ and $\MU_2$ to coincide, it does not suffice that $T_1$ and $T_2$ are unitarily equivalent. The unitary map must also respect the cyclic vectors. It may therefore be appropriate to regard Theorem \ref{thm:ModelTheorem} as a model for the \textit{tuple} $(T,e)$, rather than just $T$. The case $\mu_0=\lambda$ corresponds precisely to the case where $e$ has unit length, and $e\in\ker T^*$. This condition determines the cyclic vector up to multiplication with a unitary scalar, so that $(T,e)$ is determined by $T$.
	\end{remark}
	
	\begin{remark}\label{remark:InterchangeableNotation}
		We interchangeably use the notations $M_{\mu_{0},\ldots,\mu_{m-1}}$ and $ M_{T,e}$, in place of $M_{\MU}$. Note that $\DC_\lambda^2=H^2$, since polynomials are dense in $H^2$. Hence the notation $M_\lambda:H^2\to H^2$. Similarly, if $\mu$ is a finite positive measure on $\TB$, then $\DC_{\lambda,\mu}^2$ coincides with the space of analytic functions $f:\DB\to\CB$ for which 
		\[
		\frac{1}{\pi}\int_{\DB}|f'(z)|^2P_\mu(z)\diff A(z)<\infty,
		\]
		e.g. \cite[Corollary 7.3.4]{ElFallah-Kellay-Mashreghi-Ransford2014:APrimerOnTheDirichletSpace}.
	\end{remark}
	\begin{remark}
		The above model is consistent in the following sense: Suppose that $\NU\in\left(\DC'\right)^m$ is a normalized allowable $m$-tuple. By Theorem \ref{thm:ModelTheorem}, $M_{\NU}$ is an $m$-isometry with cyclic unit vector $1$. Moreover, there exists an allowable tuple $\MU$, and a unitary map $U:\DC_{\NU}^2\to\DC_{\MU}^2$, such that $M_{\NU}=U^*M_{\MU}U$, and $U1=1$. By the final conclusion of Theorem \ref{thm:ModelTheorem}, $\NU=\MU$.
	\end{remark}
	
	An extensive study of $m$-isometries was undertaken by Agler and Stankus in \cite{Agler-Stankus1995:m-IsometricTransformationsOfHilbertSpaceI,Agler-Stankus1995:m-IsometricTransformationsOfHilbertSpaceII,Agler-Stankus1995:m-IsometricTransformationsOfHilbertSpaceIII}. One of their main results is \cite[Theorem 3.23]{Agler-Stankus1995:m-IsometricTransformationsOfHilbertSpaceI}, which is also a model theorem for cyclic $m$-isometries. However, their model is quite different from ours, even in the case of Proposition \ref{prop:2Isometries}. In fact, a motivation for this paper has been that the Agler--Stankus model seems less open to function theoretic exploitations. We will briefly compare the two models in Section \ref{sec:AglerStankus}.
	
	In Section \ref{sec:MorePropertiesOfDirichletIntegrals}, we derive some additional properties of Dirichlet integrals. These are then used for studying allowable $m$-tuples in Section \ref{sec:AllowableMTuples}:
	
	From previous work by Richter \cite{Richter1991:ARepresentationTheoremForCyclicAnalyticTwo-Isometries}, and Agler--Stankus \cite{Agler-Stankus1995:m-IsometricTransformationsOfHilbertSpaceI}, it is essentially known that $\mu_{m-1}$ is a positive finite measure, whenever the $m$-tuple $(\mu_0,\ldots,\mu_{m-1})$ is allowable. An extension of this is Proposition \ref{prop:PropertiesOfmuk} below, which gives a priori lower bounds on the regularity of each $\mu_n$. The most obvious deficiency of the present paper is that this result is not accompanied by a useful characterization of allowable $m$-tuples. However, we are able to present some examples of sufficient (and insufficient) conditions. Our results yield explicit examples of the following:
	\begin{itemize}
		\item Allowable tuples for which $M_{\MU}:\DC_{\MU}^2\to \DC_{\MU}^2$ is norm-expanding (Theorem \ref{thm:SufficientPositiveMeasures}).
		\item Allowable tuples for which $M_{\MU}:\DC_{\MU}^2\to \DC_{\MU}^2$ is not norm-expanding (Remark \ref{remark:NormExpansion}). For $m=2$, such examples are known to not exist, e.g. \cite[Lemma 8.2.3]{ElFallah-Kellay-Mashreghi-Ransford2014:APrimerOnTheDirichletSpace}.
		\item Allowable tuples for which $\DC_{ \mu_0,0}(f)+\DC_{\mu_{m-1}}(f)$ does not control $\|f\|_{\MU}^2$ (Example \ref{ex:OuterTermsDoNotControlTheMiddle}). This contrasts to the theory of Sobolev spaces, where similar estimates are standard, e.g. \cite[Chapter 1]{Mazya2011:SobolevSpacesWithApplicationsToEllipticPartialDifferentialEquations}.
		\item Allowable tuples where, for some $n$ and $f$, the Dirichlet integral $\DC_{\mu_n,n}(f)$ is indeed conditionally convergent (Example \ref{ex:DirichletIntegralIsConditionallyConvergent}).
		\item Tuples (not allowable) for which $M_{\MU}$ is an unbounded, densely defined $2$-isometry (Example \ref{ex:Unbounded2Isometry}).
	\end{itemize}
	It seems to be of interest that, for all allowable $m$-tuples that are explicitly described in this paper, it holds that $\DC_{ \lambda,m-2}(f)\lesssim \|f\|_{\MU}^2$, provided that $\mu_{m-1}$ is non-vanishing.
	
	Section \ref{sec:ConcludingRemarks} contains some concluding remarks.
	
	\section{Notation and preliminaries}\label{sec:Preliminaries}
	We use the standard notation $\ZB$, $\RB$, and $\CB$ for the respective sets of integers, real numbers, and complex numbers. In addition we write $\ZB_{\ge x}=\{n\in\ZB;n\ge x\}$, $\DB=\{z\in\CB;|z|<1\}$ and $\TB=\{\zeta\in\CB;|\zeta|=1\}$. We let $\diff \lambda$ and $\diff A$ respectively signify integration with respect to arc length measure and area measure on $\CB$. By $\delta_z$, we denote a unital point mass at $z\in\CB$, while $\delta_{x,y}=\delta_{0}(\{x-y\})$ is Kronecker's delta. We will also use the lattice operators $x\wedge y=\min\{x,y\}$ and $x\vee y=\max\{x,y\}$, where $x,y\in\RB$.
	
	For $n,j\in\ZB_{\ge 0}$ we let $\binom{n}{j}=\frac{n!}{j!(n-j)!}$ denote the standard binomial coefficients. We also define the Pochhammer symbols $(n)_j=j!\binom{n}{j}$. Note that if $j$ is fixed, then $\binom{n}{j}$ is a polynomial in $n$, its degree $j$. We adopt the conventions that $0!=1$, and $\binom{n}{j}=0$ whenever $j\in \ZB_{<0}\cup\ZB_{>n}$. We will frequently use the identity
	\begin{equation}\label{eq:BinomialIdentity}
	\binom{n+1}{j}=\binom{n}{j}+\binom{n}{j-1}.
	\end{equation}
	By our last convention, \eqref{eq:BinomialIdentity} holds for $n\in\ZB_{\ge 0}$ and $j\in\ZB$.
	
	Given two parametrized sets $\{A_i\}_{i\in I},\{B_i\}_{i\in I}\subset (0,\infty)$, we write $A_i\lesssim B_i$ for $i\in I$ to indicate the existence of $C\in (0,\infty)$ such that $A_i\le C B_i$ whenever $i\in I$. We refer to $C$ as a \textit{bound}, and say that $A_i$ is \textit{bounded} or \textit{controlled} by $B_i$. Typically, $I$ will be implicit from the context. We then write $A_i\lesssim B_i$. If $A_i\lesssim B_i$ and $B_i\lesssim A_i$, then we write $A_i\approx B_i$, and say that $A_i$ is \textit{comparable} to $B_i$.
	
	Denote by $\DC$ the Fréchet space of smooth functions on $\TB$, equipped with the seminorms
	\begin{equation}
	\|f\|_\alpha = \left(\sum_{k\in\ZB}|\hat f(k)|^2(1+|k|)^{\alpha}\right)^{1/2},\quad \alpha>0.
	\end{equation}
	Here $\hat f(k)=\frac{1}{2\pi}\int_\TB f(\zeta)\conj{\zeta}^k\diff \lambda(\zeta)$. A function $f:\TB\to\CB$ belongs to $\DC$ if and only if $f\in L^1(\TB,\diff\lambda)$ and $|\hat f(k)|\lesssim (1+|k|)^{-N}$ whenever $N\in\ZB_{\ge 0}$, i.e. $\|f\|_\alpha <\infty $ for every $\alpha>0$. Note that if $f\in\DC$, then its Fourier partial sums given by $S_nf(\zeta)=\sum_{|k|\le n}\hat f (k) \zeta^k$ converge to $f$ in $\DC$ as $n\to\infty$.
	
	The topological dual of $\DC$, i.e. the space of distributions on $\TB$, is denoted by $\DC'$. Given $\mu\in\DC'$ we define $\hat{\mu}(k)=\mu(\zeta^{-k})$. Since $\TB$ is compact, continuity of $\mu$ implies that $|\hat{\mu}(k)|\lesssim(1+|k|)^N$ for some $N\in\ZB_{\ge 0}$. The smallest such $N$ we call the \textit{(Fourier-)order} of $\mu$. Since trigonometric polynomials are dense in $\DC$, and any $\mu\in\DC'$ has finite order, it holds that $\mu\left(\conj{f}\right)=\sum_{k\in\ZB}\hat \mu (k)\conj{\hat f(k)}$, where the series is absolutely convergent. Conversely, this series defines an element $\mu \in \DC'$ whenever $(\hat \mu (k))_{k\in\ZB}$ is a sequence satisfying $|\hat{\mu}(k)|\lesssim(1+|k|)^N$ for some $N\in\ZB_{\ge 0}$.
	
	Given $\mu\in\DC'$ we denote by $P_\mu$ its Poisson extension to $\DB$, i.e.
	\[
	P_\mu (z):=\mu(P_z)=\sum_{k=0}^\infty \hat \mu (k) z^k+\sum_{k=1}^\infty \hat \mu (-k)\conj{z}^k,\quad z\in\DB,
	\]
	where
	\[
	P_z(\zeta):=\frac{1-|z|^2}{|\zeta-z|^2}=\sum_{k=0}^\infty \left(\conj{\zeta}z\right)^k+\sum_{k=1}^\infty \left(\zeta \conj{z}\right)^k,\quad \zeta \in\TB,
	\]
	is the standard Poisson kernel with respect to $z$.	By means of Poisson extensions, we may regard $\DC'$ as the space of harmonic functions with Taylor coefficients having moderate growth. Similarly, $\DC$ is identified with the space of harmonic functions with Taylor coefficients having rapid decay. By $\DC_a\subset\DC $ we denote the subspace of analytic functions. 
	
	Some distributions $\mu\in\DC'$ can be represented as integration against a finite measure. We somewhat abusively then say that the distribution \textit{is} a finite measure, and write $\mu(f)=\int_{\TB}f\diff\mu $. 
	
	If $\mu$ is a finite positive measure, then it's Poisson extension $P\mu$ satisfies
	\begin{equation}\label{eq:PoissonKernelEstimates}
	\mu(\TB)(1-|z|^2)\lesssim P_\mu(z),\quad\textnormal{and}\quad P_\mu(z)\lesssim \mu(\TB)(1-|z|^2)^{-1}.
	\end{equation}
	This follows from the Poisson kernel estimates
	\[
	\frac{1-|z|}{1+|z|}\le\frac{1-|z|^2}{|\zeta-z|^2}\le\frac{1+|z|}{1-|z|}.
	\]
	
	A harmonic function on $\DB$ is positive if and only if it is the Poisson extension of a finite positive measure, c.f. the proof of \cite[Chapter I, Theorem 3.5]{Garnett2007:BoundedAnalyticFunctions}. By considering the Jordan decomposition of a signed measure, one obtains that a real-valued harmonic function on $\DB$ is the Poisson extension of a real-valued finite measure if and only if it can be written as the difference between two positive harmonic functions.
	
	If $\mu$ is a finite positive measure, then it is \textit{positive as a distribution}, i.e. $\mu(f)\ge 0$ for each $f\in\DC$ with $f\ge 0$. Any positive distribution is in fact a finite positive measure, e.g. \cite[Theorem 2.1.7]{Hormander1990:TheAnalysisOfLinearPartialDifferentialOperatorsI}. By considering convolutions with Fejér kernels, any $f\in\DC$ with $f\ge 0$ may be approximated in $\DC$ by a sequence of polynomials $p_n\ge0$. By the Fejér--Riesz theorem, $p_n=|g_n|^2$, where $g_n\in\DC_a$. Hence, $\mu\in\DC'$ is positive if and only if $\mu(|g|^2)\ge 0$ for $g\in\DC_a$.

	The Hardy space $H^2$ is defined as
	\[
	H^2=\{f\in L^2(\TB,\diff\lambda); \hat f(k)=0 \textnormal{ whenever } k<0\}.
	\]
	The Poisson extension operator restricted to $H^2$ is a unitary operator into the space of functions analytic on $\DB$ having square summable Taylor coefficients. The inverse of this operator is given by the identity $f(\zeta)=\lim_{r\to 1^{-}}P_f(r\zeta)$, valid for $\lambda$-a.e. $\zeta\in\TB$. We will typically not distinguish $f\in H^2$ from $P_f$.
	
	A positive measure $\nu$ on $\DB$ is called a Carleson measure if 
	\[
	\|\nu\|_{CM}=\sup_{w\in\DB}\int_ \DB \frac{1-|w|^2}{|1-\conj{w} z|^2}\diff \nu(z)<\infty.
	\]
	In particular, any such measure is finite. The Carleson embedding theorem states that the above condition is equivalent to that the Hardy space $H^2$ is continuously embedded into the space $L^2 (\DB,\diff \nu)$. Specifically, if $C>0$ denotes the smallest number such that
	\[
	\int_{\DB}|f(z)|^2\diff \nu(z)\le C^2\|f\|_{H^2},\quad f\in H^2,
	\]
	then $C^2\approx \|\nu\|_{CM}$, e.g. \cite[Chapter I, Theorem 5.6]{Garnett2007:BoundedAnalyticFunctions}. 
	
	The next proposition is a slight extension of results from \cite{Agler-Stankus1995:m-IsometricTransformationsOfHilbertSpaceI,Richter1991:ARepresentationTheoremForCyclicAnalyticTwo-Isometries}. We have essentially added an induction step. For the readers convenience we provide a proof:
	\begin{proposition}\label{prop:PropertiesOfBetak}
		Let $T\in\LC$.
		\begin{enumerate}[(i)]
			\item For $n,j\in\ZB_{\ge 0}$, it holds that
			\begin{equation}\label{eq:RecursionForBetaK}
			\beta_{n+j}(T)=\sum_{i=0}^j(-1)^{j-i}\binom{j}{i}T^{*i}\beta_n(T)T^{i}.
			\end{equation}
			In particular, if $T$ is $m$-isometric for $m\in\ZB_{\ge 1}$, then $T$ is $(m+j)$-isometric for all $j\in\ZB_{\ge 1}$. Moreover, the case $j=1$ implies that $T$ is an isometry with respect to the sesquilinear form $\langle \beta_{m-1}(T)\cdot,\cdot\rangle_\HC$.
			\item For $k,n\in\ZB_{\ge 0}$, it holds that
			\begin{equation}\label{eq:SymbolEquation}
			T^{*k}\beta_n(T)T^k=\sum_{j=0}^k \binom{k}{j} \beta_{n+j}(T).
			\end{equation}
			In particular, if $T$ is $m$-isometric for $m\in\ZB_{\ge 1}$, then $T^{*k}\beta_n(T)T^k$ is a polynomial in $k$, with operator coefficients, and degree at most $m-1-n$.
			\item If $T$ is $m$-isometric, then the operator $\beta_{m-1}(T)$ is positive on $\HC$. Moreover, if $x\in\HC$, then the distribution $\mu_x\in\DC'$ given by $\hat{\mu}_x(k)=\conj{\hat \mu _x(-k)}=\langle \beta_{m-1}(T)x,T^kx\rangle_{\HC}$ for $k\in\ZB_{\ge 0}$ is a positive measure such that
			\begin{equation}\label{eq:MYHANDSARETYPINGWORDS}
			\langle \beta_{m-1}(T)f(T)x,f(T)x\rangle = \int_{\TB} |f|^2 \diff\mu_x
			\end{equation}
			whenever $f\in\DC_a$.
		\end{enumerate}
	\end{proposition}
	\begin{proof}
		$(i)$ We consider fixed $n$, and use induction over $j$. The case $j=0$ is trivial. For $j=1$, use \eqref{eq:DefinitionBetaK} together with \eqref{eq:BinomialIdentity}:  
		\begin{multline*}
		\beta_{n+1}(T)
		=
		\sum_{i=0}^{n+1}(-1)^{n+1-i}\binom{n+1}{i}T^{*i}T^{i}
		\\
		=
		\sum_{i=0}^{n+1}(-1)^{n+1-i}\binom{n}{i-1}T^{*i}T^{i}
		+
		\sum_{i=0}^{n+1}(-1)^{n+1-i}\binom{n}{i}T^{*i}T^{i}.
		\end{multline*}
		Note that the first term in the first sum vanishes. Summing over $i-1$, rather than over $i$, the first sum equals $T^*\beta_{n}(T)T$. Similarly, the second sum equals $-\beta_{n}(T)$.  Hence, $(i)$ holds for $j=1$. 
		
		Assume now that our conclusion holds for some $j=j_0$. By the case $j=1$ we obtain that
		\begin{multline*}
		\beta_{n+j_0+1}(T)
		=
		T^*\beta_{n+j_0}(T)T-\beta_{n+j_0}(T)
		\\
		=
		\sum_{i=0}^{j_0}(-1)^{j_0-i}\binom{j_0}{i}T^{*i+1}\beta_{n}(T)T^{i+1}
		-
		\sum_{i=0}^{j_0}(-1)^{j_0-i}\binom{j_0}{i}T^{*i}\beta_{n}(T)T^{i}.			
		\end{multline*}
		Summing over $i+1$ in the first sum, and using \eqref{eq:BinomialIdentity}, we conclude that $(i)$ holds for $j=j_0+1$. This completes the induction argument.
		
		$(ii)$ Using $(i)$ , and changing the order to summation, we compute the right-hand side of \eqref{eq:SymbolEquation}:
		\begin{multline*}
		\sum_{j=0}^k \binom{k}{j}\beta_{n+j}(T)
		\\
		=
		\sum_{j=0}^k \binom{k}{j}\sum_{i=0}^j(-1)^{j-i}\binom{j}{i}T^{*i}\beta_n(T)T^{i}
		=
		\sum_{i=0}^kT^{*i}\beta_n(T)T^{i}\sum_{j=i}^k (-1)^{j-i}\binom{k}{j}\binom{j}{i}.
		\end{multline*}	
		Using the identity $\binom{k}{j}\binom{j}{i}=\binom{k}{i}\binom{k-i}{j-i}$, and summing over $j-i=j'$, the coefficient of $T^{*i}\beta_n(T)T^i$ in the above right-hand side becomes
		\[
		\sum_{j=i}^k(-1)^{j-i}\binom{k}{j}\binom{j}{i}
		=
		\binom{k}{i}\sum_{j=0}^{k-i}(-1)^{j}\binom{k-i}{j}
		=
		\left\{
		\begin{array}{ll}
		0 & \textnormal{for }i<k,\\
		1 & \textnormal{for }i=k,
		\end{array}
		\right.
		\]
		as follows by the binomial theorem. This proves $(ii)$.
		
		$(iii)$ Let $x\in\HC$. Applying $(ii)$ with $n=0$, $\|T^kx\|_\HC^2$ is a polynomial with leading coefficient $\frac{1}{(m-1)!}\langle\beta_{m-1}(T)x,x\rangle_\HC$. Clearly this must be positive.
		
		It follows from $(i)$ that $\langle\beta_{m-1}(T)T^kx,T^lx\rangle_\HC
		=
		\hat \mu_x(l-k)$, and so
		\begin{multline*}
		\langle \beta_{m-1}(T)f(T)x,f(T)x\rangle_\HC
		\\
		=
		\sum_{k,l=0}^\infty \hat f(k)\conj{\hat f(l)}\langle \beta_{m-1}(T)T^kx,T^lx\rangle_\HC
		=
		\sum_{k,l=0}^\infty \hat f(k)\conj{\hat f(l)}\hat \mu_x(l-k)
		=
		\mu_x(|f|^2).
		\end{multline*}
		By the Fejér--Riesz theorem, any positive function in $\DC$ may be approximated by functions of the form $|f|^2$ where $f\in\DC_a$. The fact that $\beta_{m-1}(T)\ge 0$ now implies that $\mu_x$ is a positive distribution, and \eqref{eq:MYHANDSARETYPINGWORDS} follows from the Riesz representation theorem.
	\end{proof}
	
	A consequence of statement $(ii)$ in the above proposition is that if $T$ is an  $m$-isometry, then $\sigma(T)\subset \clos{\DB}$, by Gelfand's formula for the spectral radius. Hence the map $\DC_a\ni f\mapsto \sum_{k=0}^\infty \hat f(k)T^k\in \LC$ is a continuous unital algebra homomorphism, cf. \cite[Proposition 1.20]{Agler-Stankus1995:m-IsometricTransformationsOfHilbertSpaceI}. Since analytic polynomials are dense in $\DC_a$, this implies that $e$ is cyclic for $T$ if and only if $\bigvee\{f(T)e;f\in\DC_a\}=\HC$. 
	
	Let $\sigma_{ap}(T)=\{z\in\CB;z-T\textnormal{ is not bounded from below}\}$, the approximate point spectrum of $T$. A slightly more careful analysis shows that if $T$ is an $m$-isometry, then $\sigma_{ap}(T)\subset \TB$, cf. \cite[Lemma 1.21]{Agler-Stankus1995:m-IsometricTransformationsOfHilbertSpaceI}. In particular, $T$ is bounded from below.
	
	\section{The model theorem}\label{sec:ModelTheorem}
	
	\begin{theorem}\label{thm:ModelTheorem}
		If $\MU\in\left(\DC'\right)^m$ is a normalized allowable $m$-tuple, then the operator $M_{\MU}$ is a bounded $m$-isometry, with $1$ as a cyclic unit vector.	Conversely, if $T\in\LC$ is an $m$-isometry with a cyclic unit vector $e$, then $\MU\in\left(\DC'\right)^m$ given by
		\[
		\hat \mu_n(k)=\conj{\hat \mu_n(-k)}=\langle \beta_{n}(T)e,T^ke\rangle_\HC,\quad k\in\ZB_{\ge 0},
		\]
		is a normalized allowable $m$-tuple, and there exists a unitary map $U:\HC\to \DC_{\MU}^2$ such that $T=U^*M_{\MU}U$, and $Ue=1$. 
		
		If $T_j:\HC_j\to\HC_j$, $j\in\{1,2\}$, are $m$-isometries with cyclic vectors $e_j$, then the associated $m$-tuples $\MU_{j}$ coincide if and only if there exists a unitary map $U:\HC_1\to\HC_2$ such that $T_1=U^*T_2U$ and $Ue_1=e_2$.
	\end{theorem}
	
	For proving the first part of this theorem, we derive a formula for weighted Dirichlet integrals:
	
	\begin{lemma}\label{lemma:DirichletIntegralAsADoubleSum}
		Let $\mu\in\DC'$, $f\in\DC_a$, and $n\in\ZB_{\ge 0}$. Then
		\begin{equation}\label{eq:DirichletIntegralAsADoubleSum}
		\DC_{\mu,n}(f)=\sum_{k,l=0}^\infty\binom{k\wedge l}{n}\hat f(k)\conj{\hat f(l)}\hat \mu(l-k).
		\end{equation}
		Moreover, the right-hand side is absolutely convergent.
	\end{lemma}
	\begin{remark}
		Recall that, by convention, if $k\wedge l<n$, then $\binom{k\wedge l}{n}=0$. Hence, the above right-hand side is equal to $\sum_{k,l=n}^\infty\binom{k\wedge l}{n}\hat f(k)\conj{\hat f(l)}\hat \mu(l-k)$.
	\end{remark}
	\begin{proof}
		We begin with the statement about absolute convergence. The binomial coefficient $\binom{k\wedge l}{n}$ is a $n$th degree polynomial in $k\wedge l$, and $|\hat{\mu}(l-k)|\lesssim \left(1+|l-k|\right)^N\lesssim (1+k\vee l)^N$, where $N$ is the order of $\mu$. Consequently,
		\begin{multline*}
		\sum_{k,l=0}^\infty\left|\binom{k\wedge l}{n}\hat f(k)\conj{\hat f(l)}\hat \mu(l-k)\right|
		\lesssim 
		\sum_{k,l=0}^\infty(1+k\wedge l)^{n}(1+k\vee l)^N|\hat f(k)\conj{\hat f(l)}|
		\\
		\le
		\left(\sum_{k=0}^\infty (1+k)^{n\vee N}|\hat f(k)|\right)\left(\sum_{l=0}^\infty (1+l)^{n\vee N}|\hat f(l)|\right).
		\end{multline*}
		The right-hand side is absolutely convergent because of the rapid decay of $\left(\hat f(k)\right)_{k\ge 0}$.
		
		We now prove the equality \eqref{eq:DirichletIntegralAsADoubleSum}. We consider only the case $n\ge 1$. The case $n=0$ is treated similarly. Note that 
		\[
		|f^{(n)}(z)|^2=\sum_{k,l=0}^\infty(k)_n(l)_n\hat f (k)\conj{\hat f(l)}z^{k-n}\conj{z}^{l-n}.
		\] 
		
		Let $r\in(0,1)$. Since the power series of $P_\mu$ is uniformly convergent on the disc $r\DB$, we may interchange summation and integration in order to obtain
		\begin{multline*}
		\int_{r\DB}|f^{(n)}(z)|^2P_{\mu}(z)(1-|z|^2)^{n-1}\diff A(z)
		\\
		=
		\sum_{k,l=n}^\infty(k)_n(l)_n\hat f(k)\conj{\hat f(l)}
		\sum_{j\in\ZB}\hat \mu(j)\int_{r\DB}z^{k-n}\conj{z}^{l-n}z^*(j)(1-|z|^2)^{n-1}\diff A(z),
		\end{multline*}
		where $z^*(j)=\conj{z^*(-j)}=z^j$ for $j\in\ZB_{\ge 0}$. Using polar coordinates, one obtains that the integral in the above right-hand side vanishes, unless $j=l-k$. For $j=l-k$ one obtains that
		\begin{multline*}
		\int_{r\DB}z^{k-n}\conj{z}^{l-n}z^*(l-k)(1-|z|^2)^{n-1}\diff A(z)
		\\
		=2\pi\int_{\rho=0}^r\rho^{2k\vee l-2n+1}(1-\rho^2)^{n-1}\diff \rho
		=\pi\int_{\rho=0}^{\sqrt{r}} \rho^{k\vee l-n}(1-\rho)^{n-1}\diff \rho,
		\end{multline*}
		by the change of variables $\rho^2=\rho'$. By the monotone convergence theorem,
		\begin{equation}\label{eq:ApproximateBFunction}
		\int_{\rho=0}^{\sqrt{r}}\rho^{k\vee l-n}(1-\rho)^{n-1}\diff \rho \nearrow
		\int_{\rho=0}^1\rho^{k\vee l-n}(1-\rho)^{n-1}\diff \rho\quad\textnormal{as }r\to 1^{-}.
		\end{equation}
		By well-known properties of the Euler $B$-function, the last integral equals 
		\[
		\frac{(k\vee l-n)!(n-1)!}{(k\vee l)!}=\frac{(n-1)!}{(k\vee l)_n}.
		\]
		Using the fact that $(k\vee l)_n=(k)_n\vee(l)_n$, we now have that
		\begin{align*}
		\DC_{\mu,n}(f)
		= {}&
		\lim_{r\to 1^{-}}\frac{1}{n!(n-1)!\pi }\int_{r\DB}|f^{(n)}(z)|^2P_{\mu}(z)(1-|z|^2)^{n-1}\diff A(z)
		\\
		= {}& 
		\lim_{r\to 1^{-}}\frac{1}{n!(n-1)! }\sum_{k,l=0}^\infty(k)_n(l)_n\hat f(k)\conj{\hat f(l)}
		\hat \mu(l-k)
		\int_{\rho=0}^{\sqrt{r}}\rho^{k\vee l-n}(1-\rho)^{n-1}\diff \rho
		\\
		= {}&
		\sum_{k,l=0}^\infty\frac{(k\wedge l)_n}{n!}\hat f(k)\conj{\hat f(l)}
		\hat \mu(l-k) 
		\\
		= {}&
		\sum_{k,l=0}^\infty\binom{k\wedge l}{n}\hat f(k)\conj{\hat f(l)}
		\hat \mu(l-k).
		\end{align*}
		We already proved that the resulting series is absolutely convergent, and since the limit \eqref{eq:ApproximateBFunction} is increasing, the second to last equality is justified by the dominated convergence theorem.
	\end{proof}
	
	Lemma \ref{lemma:DirichletIntegralAsADoubleSum} yields the following result:
	
	\begin{proposition}\label{prop:MzPlaysNicelyWithDirichletIntegrals}
		If $n\in\ZB_{\ge 0}$ and $\mu\in\DC'$, then
		\begin{equation}\label{eq:MzPlaysNicelyWithDirichletIntegrals1}
		\DC_{\mu,n}(zf)=\DC_{\mu,n}(f)+\DC_{\mu,n-1}(f).
		\end{equation}
	\end{proposition}
	\begin{proof}
		Note that $zf(z)=\sum_{k=1}^\infty \hat f(k-1)z^k$. Using Lemma \ref{lemma:DirichletIntegralAsADoubleSum}, followed by a shift of summation indices and the binomial identity \eqref{eq:BinomialIdentity}, we compute
		\begin{multline*}
		\DC_{\mu,n}(zf)
		=
		\sum_{k,l=n-1}^\infty\binom{k\wedge l+1}{n}\hat f(k)\conj{\hat f(l)}\hat \mu(l-k)
		\\
		=
		\sum_{k,l=n-1}^\infty\left[\binom{k\wedge l}{n}+\binom{k\wedge l}{n-1}\right]\hat f(k)\conj{\hat f(l)}\hat \mu(l-k)
		=
		\DC_{\mu,n}(f)+\DC_{\mu,n-1}(f).
		\end{multline*}
	\end{proof}
	
	We are now prepared to prove the first part of Theorem \ref{thm:ModelTheorem}: Assume that $\MU$ is a normalized allowable $m$-tuple. By \eqref{eq:AllowableSequence}, we have that $\|M_zf\|_{\MU}=0$ whenever $\|f\|_{\MU}=0$, and so $M_{\MU}:\DC_a/\KC_{\MU}\to\DC_a/\KC_{\MU}$ is a well-defined bounded operator. 
	
	By an application of the dominated convergence theorem, Lemma \ref{lemma:DirichletIntegralAsADoubleSum} implies that if $\mu\in\DC'$, and if $f\in\DC_a$ has Fourier partial sums $(s_Nf)_{N\in\ZB_{\ge 0}}$, then
	\begin{equation}\label{eq:ApproximationByFourierPartialSums}
	\DC_{\mu,n}(f)=\lim_{N\to\infty}\DC_{\mu,n}(s_Nf),\quad \textnormal{hence}\quad \|f\|_{\MU}=\lim_{N\to\infty}\|s_Nf\|_{\MU},
	\end{equation}
	so that when we take the completion of $\DC_a/\KC_{\MU}$ with respect to $\|\cdot\|_{\MU}$, it suffices to consider analytic polynomials. Since any analytic polynomial is in the span of $\{M_z^k1;k\in\ZB_{\ge 0}\}$, $1$ is a cyclic vector for $M_{\MU}$. Moreover, $\|1\|_{\MU}^2=\hat \mu_{0}(0)=1$, since $\MU$ is normalized.
	
	We prove that $M_{\MU}$ is $m$-isometric by using Proposition \ref{prop:MzPlaysNicelyWithDirichletIntegrals} iteratively. One iteration yields
	\begin{equation*}
	\DC_{\mu,n}(z^2f)-2\DC_{\mu,n}(zf)+\DC_{\mu,n}(zf)
	=
	\DC_{\mu,n-1}(zf)-\DC_{\mu,n-1}(f)=\DC_{\mu,n-2}(f).
	\end{equation*}
	An induction argument shows that in general
	\[
	\sum_{j=0}^{m}(-1)^{m-j}\binom{m}{j}\DC_{\mu,n}(z^{m-j}f)=\DC_{\mu,n-m}(f).
	\]
	Applying this identity to each term in \eqref{eq:DefWeightedDirichletNorm} yields that
	\[
	\sum_{j=0}^{m}(-1)^{m-j}\binom{m}{j}\|M_z^{m-j}f\|_{\MU}^2=0,
	\]
	i.e. $M_z$ is $m$-isometric on $\DC_a$ with respect to $\|\cdot\|_{\MU}^2$. Hence $M_{\MU}$ is an $m$-isometry.
	
	To prove the second part of Theorem \ref{thm:ModelTheorem}, let $T\in\LC$ be an $m$-isometry with a cyclic unit vector $e$, and recall that the corresponding tuple $\MU$ is defined by 
	\begin{equation}\label{eq:DefinitionOfDistributions}
	\hat \mu_n(k)=\conj{\hat \mu_n(-k)}=
	\langle \beta_n(T)e,T^ke\rangle_\HC,\quad k\in\ZB_{\ge 0}.
	\end{equation}
	It is clear that $\MU$ is normalized. 
	
	For $f\in\DC_a$, we have that $\|f\|_{T,e}^2=\sum_{n=0}^{m-1}\DC_{\mu_n,n}(f)$. Our main technical lemma is the following:	
	\begin{lemma}\label{lemma:SmoothFunctionalCalculusIsometry}
		For any $f\in\DC_a$, it holds that
		\begin{equation}\label{eq:SmoothFunctionalCalculusIsometry}
		\|f\|_{T,e}^2= \|f(T)e\|_\HC^2 .
		\end{equation}
	\end{lemma}
	\begin{proof}
		The statement is equivalent to the claim that
		\begin{equation}\label{eq:Claim}
		f,g\in\DC_a\Rightarrow \langle f,g\rangle _{T,e} = \langle f(T)e,g(T)e\rangle_\HC .
		\end{equation}
		The Fourier partial sums $(s_Nf)_{N=0}^\infty$ converge to $f$ in $\DC_a$ as $N\to \infty$. By continuity of the functional calculus, it follows that $\langle s_Nf(T)e,s_Ng(T)e\rangle_\HC\to \langle f(T)e,g(T)e\rangle_\HC$. By \eqref{eq:ApproximationByFourierPartialSums}, $\langle s_Nf,s_Ng\rangle_{T,e}\to \langle f,g\rangle_{T,e}$. Hence it suffices to verify \eqref{eq:Claim} for polynomials. By linearity, we may restrict ourselves to monomials, $f(z)=z^k$ and $g(z)=z^l$, and by symmetry, we can assume that $l\ge k$. Applying Lemma \ref{lemma:DirichletIntegralAsADoubleSum}, and Proposition \ref{prop:PropertiesOfBetak} $(ii)$ with $n=0$, we obtain
		\begin{multline*}
		\langle z^k,z^l\rangle_{T,e}
		=
		\sum_{j=0}^{m-1}\binom{k}{j}\hat{\mu}_j(l-k)
		\\
		= 
		\sum_{j=0}^{\infty}\binom{k}{j}\langle\beta_j(T)e,T^{l-k}e\rangle_\HC
		=
		\langle T^{*k}\beta_0(T)T^ke,T^{l-k}e\rangle_\HC
		=
		\langle T^ke,T^le\rangle_\HC.
		\end{multline*}
		This completes the proof of Lemma \ref{lemma:SmoothFunctionalCalculusIsometry}.
	\end{proof}
	
	Lemma \ref{lemma:SmoothFunctionalCalculusIsometry} implies that $\|f-g\|_{\MU}=0$ if and only if $f(T)e=g(T)e$. This yields a well-defined isometric operator $\tilde U:f(T)e\mapsto f\in\DC_a/\KC_{T,e}$. Since $e$ is a cyclic vector for $T$, $\tilde U$ extends to a uniquely determined unitary operator $U:\HC\to\DC_{\MU}^2$. The fact that $T=U^*M_{\MU}U$ and $Ue=1$ follows from the definition of $\tilde U$. 
	
	To prove the final part of Theorem \ref{thm:ModelTheorem}, let $T_j:\HC_j\to\HC_j$, $j\in\{1,2\}$, be $m$-isometries with cyclic vectors $e_j$. If there exists a unitary map $U:\HC_1\to\HC_2$ such that $T_1=U^*T_2U$ and $Ue_1=e_2$, then a straightforward verification shows that $\MU_1=\MU_2$. Conversely, if $\MU_1=\MU_2$, then $M_{T_1,e}=M_{T_2,e}$, hence $T_1=U_1^*M_{T_1}U_1= U_1^*M_{T_2}U_1=U_1^*U_2T_2U_2^*U_1$, and $U_2^*U_1e_1=U_2^*1=e_2$. Since $U_2^*U_1:\HC_1\to\HC_2$ is unitary, this concludes the proof of Theorem \ref{thm:ModelTheorem}.
	
	\section{The Agler--Stankus model}\label{sec:AglerStankus}
	
	The authors of \cite{Agler-Stankus1995:m-IsometricTransformationsOfHilbertSpaceI} regard $\DC$ and $\DC_a$ as spaces of smooth functions on $\TB$. Hence, we introduce the notation $M_{e^{i\theta}}$ as a complement to $M_z$. 
	
	Define the operator $D:\DC'\to\DC'$ by $\widehat{D\mu}(k)=|k|\hat \mu(k)$. If $\varphi\in\DC_a$, then $D\varphi(e^{i\theta})=\frac{1}{i}\frac{\diff }{\diff \theta}\varphi(e^{i\theta})$. Furthermore, let
	\[
	D^{(n)}=D\cdot (D-1)\cdot\ldots\cdot(D-n+1),\quad n\in\ZB_{\ge 0}.
	\]
	A \textit{distribution Toeplitz operator} (DTO) is a map $L:\DC_a\to\DC_a'$ given by 
	\begin{equation}
	L(\varphi)(\psi)=\sum_{n=0}^{m-1}\beta_n((D^{(n)}\varphi)\psi)
	\end{equation}
	where $m\in\ZB_{\ge 1}$, $\beta_0,\ldots,\beta_{m-1}\in\DC'$, and $\beta_{m-1}\ne 0$. We refer to $m$ as the \textit{order} of $L$. 
	
	\begin{remark}
		Our definitions deviate by convention from \cite{Agler-Stankus1995:m-IsometricTransformationsOfHilbertSpaceI} in two respects: First, if $\varphi\in\DC_a$, then $D\varphi$ denotes the same thing in both papers, whereas if $\conj{\varphi}\in\DC_a$, then $D\varphi$ in our notation is the negative of the same expression in \cite{Agler-Stankus1995:m-IsometricTransformationsOfHilbertSpaceI}. Second, we say that the above DDO has order $m$, rather than $(m-1)$.
	\end{remark}
	
	Let $A$ be a DTO of order $m\ge 2$. If there exists $c>1$ such that
	\begin{equation}\label{eq:ADOPositivityCondition}
	A-c^{-2}e^{-i\theta}Ae^{i\theta}\ge 0,
	\end{equation}
	i.e. $A(\varphi)(\conj{\varphi})-c^{-2}A(M_{e^{i\theta}}\varphi)(\conj{M_{e^{i\theta}}\varphi})\ge 0$ for $\varphi\in\DC_a$,
	then we call $A$ an \textit{analytic Dirichlet operator} (ADO) of order $m$. A DTO of order $1$ is an ADO if $\beta_0$ is a positive measure.
	
	Let $A$ be an ADO, and define the sesquilinear form
	\begin{equation}
	[\varphi,\psi]_A=A(\varphi)(\conj{\psi}),
	\end{equation}
	on $\DC_a$.	This form is positive semi-definite (c.f. \cite[Lemma 3.18]{Agler-Stankus1995:m-IsometricTransformationsOfHilbertSpaceI}), and we denote the corresponding seminorm by $\|\cdot\|_A$. Let $\KC_A=\mathrm{Ker}\|\cdot\|_A$, and define $H_A^2$ as the completion of  $\DC_a/\KC_A$ with respect to $\|\cdot\|_A$.
	
	\begin{lemma}[{\cite[Lemma 3.18]{Agler-Stankus1995:m-IsometricTransformationsOfHilbertSpaceI}}]
		If $A$ is an \textnormal{ADO}, then $M_{e^{i\theta}}$ is a well-defined operator on $\DC_a/\KC_A$, and extends uniquely to a bounded linear operator on the space $H^2_A$.
	\end{lemma}
	
	Let $M_A:H_A^2\to H_A^2$ denote the uniquely determined bounded linear extension of $M_{e^{i\theta}}:\DC_a/\MC_A\to\DC_a/\MC_A$.
	
	\begin{theorem}[{\cite[Theorem 3.23]{Agler-Stankus1995:m-IsometricTransformationsOfHilbertSpaceI}}]\label{thm:AglerStankus}
		Let $m\in\ZB_{\ge 1}$. If $T$ is a bounded $m$-isometry on a Hilbert space $\HC$, and if $e\in\HC$ is a cyclic vector for $T$, then there exists a unique \textnormal{ADO} $A$ of order $m$, and a unitary map $V:\HC\to H_A^2$, such that $T=V^*M_AV$ and $Ve = 1$. Conversely, if $A$ is an \textnormal{ADO} of order $m$, then $M_{A}:H_A^2\to H_A^2$ is a bounded $m$-isometry, with cyclic vector $1$.
	\end{theorem}
	
	Let $T$ be an $m$-isometry with a cyclic vector $e$. Applying Theorems \ref{thm:ModelTheorem} and \ref{thm:AglerStankus} we obtain that
	\[	
	\langle z^m,z^n\rangle_{\DC_{T,e}^2}
	=
	\langle T^me,T^ne\rangle_{\HC}
	=
	[e^{im\theta},e^{in\theta}]_A.
	\]
	This is of course not unexpected, since both theorems yield operator models of the same object. However, the relation between the two models is quite complicated, as is illustrated by the following relation between $\mu_0,\mu_1$ and $\beta_0,\beta_1$, valid in the case of $2$-isometries:
	\[
	\beta_1=\mu_1,\quad\textnormal{while}\quad \beta_0=\mu_0-(1-P)(D\mu_1),
	\]
	where $P$ denotes the analytic projection on $\DC'$, i.e. $(P\mu)(\conj{f})=\sum_{k=0}^\infty \hat \mu(k)\conj{\hat f(k)}$. These relations are observed in \cite{Agler-Stankus1995:m-IsometricTransformationsOfHilbertSpaceI}, and the respective distributions $\mu_0$ and $\mu_1$ are denoted the \textit{intercept} and \textit{slope} of the pair $(T,e)$, but this direction is not investigated for orders higher than $2$.
	
	A similarity between the Theorems \ref{thm:ModelTheorem} and \ref{thm:AglerStankus} is that cyclic higher order isometries are characterized in terms of a number of parameters, $\mu_0,\ldots,\mu_{m-1}$ and $\beta_0,\ldots,\beta_{m-1}$ respectively. However, while \cite{Agler-Stankus1995:m-IsometricTransformationsOfHilbertSpaceI} investigates ADOs, which are aggregate objects, our approach puts more focus on the individual parameters. A notable advantage of this is that the only structure of $\beta_0,\ldots,\beta_{m-1}$ obtained in \cite{Agler-Stankus1995:m-IsometricTransformationsOfHilbertSpaceI} is that $\beta_{m-1}$ is a positive measure, whereas we are able to give a priori regularity estimates for each element of an allowable $m$-tuple (Propositions \ref{prop:PropertiesOfmuk} and \ref{prop:PropertiesOfmukImproved}). Another advantage of our approach is that, even though we do not obtain a characterization of allowable $m$-tuples, we are able use Theorem \ref{thm:ModelTheorem} in order to construct explicit examples of allowable $m$-tuples, hence of higher order isometries.
	
	\section{Some more properties of $\DC_{\mu,n}(f)$}\label{sec:MorePropertiesOfDirichletIntegrals}
	
	Recall the operator $D:\DC'\to\DC'$, defined by $\widehat{D\mu}(k)=|k|\hat{\mu}(k)$ in the previous section. If we think of $M_z$ as a ``forward'' shift, then the next result may be viewed as a ``backward'' shift analogue of Proposition \ref{prop:MzPlaysNicelyWithDirichletIntegrals}:
	
	\begin{proposition}\label{prop:DifferentiationPlaysWithDirichletIntegral}
		Let $\mu\in\DC'$ and $n\in\ZB_{\ge 0}$. Then
		\[
		\DC_{\mu,n}(f')=(n+1)(n+2)\DC_{\mu,n+2}(f)+(n+1)^2\DC_{\mu,n+1}(f)+(n+1)\DC_{D\mu,n+1}(f).
		\]
	\end{proposition}
	\begin{proof}
		Writing $k\vee l=k\wedge l+|l-k|$, we observe that 
		\[
		(1+k)(1+l)=(1+k\wedge l)(1+k\vee l)=(1+k\wedge l)^2+(1+k\wedge l)|l-k|.
		\]
		We also need the identities 
		\[
		\binom{k\wedge l}{n}(1+k\wedge l)=(n+1)\binom{k\wedge l+1}{n+1}.
		\]
		and
		\[
		\binom{k\wedge l}{n}(1+k\wedge l)^2=(n+1)(n+2)\binom{k\wedge l+1}{n+2}+(n+1)^2\binom{k\wedge l+1}{n+1},
		\]
		Using Lemma \ref{lemma:DirichletIntegralAsADoubleSum}, we compute
		\begin{align*}
		\DC_{\mu,n}(f')
		={}&
		\sum_{k,l=n}^\infty \binom{k\wedge l}{n}(1+k)(1+l)\hat f(k+1)\conj{\hat f(l+1)}\hat{\mu}(l-k)
		\\
		={}&
		(n+1)(n+2)\sum_{k,l=n}^\infty\binom{k\wedge l+1}{n+2}\hat f(k+1)\conj{\hat f(l+1)}\hat{\mu}(l-k)
		\\
		&+			
		(n+1)^2\sum_{k,l=n}^\infty\binom{k\wedge l+1}{n+1}\hat f(k+1)\conj{\hat f(l+1)}\hat{\mu}(l-k)
		\\
		&+
		(n+1)\sum_{k,l=n}^\infty\binom{k\wedge l+1}{n+1}\hat f(k+1)\conj{\hat f(l+1)}|l-k|\hat{\mu}(l-k).
		\end{align*}
		By a shift of summation indices, and another application of Lemma \ref{lemma:DirichletIntegralAsADoubleSum}, this is equal to the right-hand side of the desired identity.
	\end{proof}
	
	By Lemma \ref{lemma:DirichletIntegralAsADoubleSum}, $\DC_{ \lambda,n}(f)=\sum_{k=0}^{\infty}\binom{k}{n}|\hat f(k)|^2$. Clearly 
	\[
	|\hat f(n)|^2+\DC_{ \lambda,n+1}(f)\approx \DC_{ \lambda,n}(f)+\DC_{ \lambda,n+1}(f).
	\]
	Using this observation, an application of Proposition \ref{prop:DifferentiationPlaysWithDirichletIntegral} to $\mu=\lambda$ yields the following result:
	
	\begin{proposition}\label{prop:GeneralizedLittlewood--PaleyTheorem}
		Let $n\in\ZB_{\ge 0}$. Then
		\[
		\DC_{ \lambda,n}(f')\approx |\hat f(n+1)|^2+\DC_{ \lambda,n+2}(f)\approx \DC_{ \lambda,n+1}(f)+\DC_{ \lambda,n+2}(f),\quad f\in\DC_a.
		\]
	\end{proposition}
	
	For $n=0$, the above result is just the classical Littlewood-Paley theorem, e.g. \cite[Chapter VI, Lemma 3.2]{Garnett2007:BoundedAnalyticFunctions}, applied to $f'$. For $n\ge 1$, it is a well-known fact from the theory of standard weighted Bergman spaces, e.g. \cite[Proposition 1.11]{Hedenmalm-Korenblum-Zhu2000:TheoryOfBergmanSpaces}.
	
	Our next goal is to relate $\mu$-weighted Dirichlet integrals to $\lambda$-weighted ones. In order to achieve the sufficient generality, we also need the following:
	
	\begin{proposition}\label{prop:PmudAIsCarlesonMeasure}
		Let $\mu$ be a finite positive Borel measure on $\TB$. Then the measure $\nu$ given by $\diff\nu=P_\mu\diff A$ is a Carleson measure. Moreover, $\|\nu\|_{CM}\lesssim \mu (\TB)$.
	\end{proposition}
	\begin{proof}
		Let $f_w(z)=\log \frac{1}{1-\conj{w}z} = \sum_{k=1}^\infty \frac{\conj{w}^k}{k} z^{k}$. By the monotone convergence theorem,
		\[
		\frac{|w|^2}{\pi}\int_ \DB \frac{1}{|1-\conj{w} z|^2}P_\mu (z)\diff A(z) 
		=
		\DC_{\mu,1}(f_w).
		\]
		Hence, we need to verify that $\DC_{ \mu,1}(f_w)\lesssim \frac{|w|^2\mu(\TB)}{1-|w|^2}$. 
		
		Let $\DC_{\zeta,1}(f)=\DC_{ \delta_\zeta,1}(f)$. By Fubini's theorem, $\DC_{\mu,1}(f)=\mu(\DC_{\zeta,1}(f))$. Therefore, it suffices to consider $\mu=\delta_{\zeta}$. By rotational symmetry, we may assume that $\zeta=1$. Let $r=|w|$. Using Lemma \ref{lemma:DirichletIntegralAsADoubleSum}, and geometric summation,
		\begin{multline*}
		\DC_{1,1}(f) \le \sum_{k,l=1}^\infty \frac{r^{k+l}}{k\vee l}
		=
		\sum_{k=1}^\infty \frac{r^{2k}}{k}+2\sum_{k=1}^\infty\frac{r^k}{k} \sum_{l=1}^{k-1}r^l
		\\
		=
		\log\left(\frac{1}{1-r^2}\right)
		+
		\frac{2r}{1-r}\log\left(\frac{1}{1-r}\right)
		-
		\frac{2}{1-r}\log\left(\frac{1}{1-r^2}\right).
		\end{multline*}
		One can easily verify that, the above right-hand side has the adequate asymptotics as $r\to 0^+$, and as $r\to 1^-$.
	\end{proof}
	
	\begin{remark}
		In the above proof, one can also use that $z\mapsto \log\frac{1}{1-z}$ has bounded mean oscillation, and apply \cite[Chapter VI, Lemma 3.3 and Theorem 3.4]{Garnett2007:BoundedAnalyticFunctions}. I owe this observation to Alexandru Aleman. However, the above calculation will be reused in Example \ref{ex:OuterTermsDoNotControlTheMiddle}.
	\end{remark}
	
	\begin{proposition}\label{prop:RelationLambdaAndMu}
		Let $\mu$ be a non-vanishing finite positive measure on $\TB$.
		\begin{enumerate}[(i)]
			\item If $n\in\ZB_{\ge 0}$, then
			\[
			\DC_{ \lambda,n}(f)\lesssim |\hat f(n)|^2+\DC_{ \mu,n+1}(f).
			\]
			\item If $n\in\ZB_{\ge 1}$, then
			\[
			\DC_{ \mu,n}(f)\lesssim |\hat f(n)|^2+\DC_{ \lambda,n+1}(f)\approx \DC_{ \lambda,n}(f)+\DC_{ \lambda,n+1}(f).
			\]
		\end{enumerate}
	\end{proposition}
	\begin{proof}
		$(i)$ Replace $f$ with $f'$. By Proposition \ref{prop:GeneralizedLittlewood--PaleyTheorem},
		\[
		\DC_{ \lambda,n}(f')\approx |\hat f(n+1)|^2+\DC_{ \lambda,n+2}(f).
		\]
		By \eqref{eq:PoissonKernelEstimates},
		\begin{multline*}
		\DC_{ \lambda,n+2}(f)
		\approx
		\int_{\DB}|f^{(n+2)}(z)|^2(1-|z|^2)^{n+1}\diff A(z)
		\\
		\lesssim 
		\int_{\DB}|f^{(n+2)}(z)|^2P\mu(z)(1-|z|^2)^{n}\diff A(z)
		\approx
		\DC_{ \mu,n+1}(f').
		\end{multline*}
		
		$(ii)$ The proof for $n\ge 2$ is similar to the proof of $(i)$:
		\begin{multline*}
		\DC_{ \mu,n}(f)
		\approx
		\int_{\DB}|f^{(n)}(z)|^2P_\mu(z)(1-|z|^2)^{n-1}\diff A(z)
		\\
		\lesssim 
		\int_{\DB}|f^{(n)}(z)|^2(1-|z|^2)^{n-2}\diff A(z)
		\approx
		\DC_{ \lambda,n-1}(f')
		\approx
		|\hat f(n)|^2+\DC_{ \lambda,n+1}(f).
		\end{multline*}
		For $n=1$ we apply the Carleson embedding theorem to $P_\mu\diff A$, which is a Carleson measure by Proposition \ref{prop:PmudAIsCarlesonMeasure}:
		\[
		\DC_{ \mu,1}(f)\approx\int_{\DB}|f'(z)|^2P_\mu(z)\diff A(z)\lesssim \DC_{ \lambda,0}(f')\approx |\hat f(1)|^2+\DC_{ \lambda,2}(f).
		\]
	\end{proof}
	
	\begin{remark}
		Note that if $P_\mu\ge c>0$ on $\DB$, then trivially $\DC_{\lambda,n}(f)\le \frac{1}{c}\DC_{ \mu,n}(f)$. This means that we have gained one order, compared to the above proposition. Similarly, if $P_\mu\le c<\infty$, then $\DC_{ \mu,n}(f)\le c \DC_{ \lambda,n}(f)$. In this case, we may have gained \textit{more} than one order, because in general, $(ii)$ does not extend to $n=0$. Indeed, if $\mu$ is a unital point mass at $1$, then $\DC_{ \mu,0}(f)=|f(1)|^2$. Hence, the corresponding estimate must fail, because the standard Dirichlet space $\DC_{ \lambda,\lambda}^2$ contains unbounded functions, e.g. \cite[Excercise 1.2.2]{ElFallah-Kellay-Mashreghi-Ransford2014:APrimerOnTheDirichletSpace}.
	\end{remark}
	
	We're now prepared to demonstrate that $M_z:\DC_a\to\DC_a$ is bounded with respect to Dirichlet integrals with positive harmonic weights:
	
	\begin{proposition}\label{prop:MzBoundedWRTDirichletIntegrals}
		Let $\mu\in\DC'$ be a finite positive measure, and $n\in\ZB_{\ge 0}$. Then
		\[
		\DC_{\mu,n}(f)\le \DC_{\mu,n}(M_zf)\lesssim |\hat f(n-1)|^2+\DC_{\mu,n}(f),\quad f\in\DC_a.
		\]
	\end{proposition}
	\begin{proof}
		The case $\mu=0$ is trivial. Moreover, the result is evident for $n=0$, and known for $n=1$, e.g. \cite[Theorem 8.1.2]{ElFallah-Kellay-Mashreghi-Ransford2014:APrimerOnTheDirichletSpace}. We consider non-vanishing $\mu$, and $n\ge 2$.
		
		The lower bound follows from Proposition \ref{prop:MzPlaysNicelyWithDirichletIntegrals}, since $\DC_{\mu,n-1}(f)\ge 0$. For the upper bound, we replace $f$ with $f'$, and note that $(M_zf')^{(n)}(z)=zf^{(n+1)}(z)+nf^{(n)}(z)$. Hence,
		\begin{align*}
		\DC_{\mu,n}(M_zf')
		\lesssim {} &
		\int_{\DB}|zf^{(n+1)}(z)|^2P_\mu(z)(1-|z|^2)^{n-1}\diff A(z)
		\\
		&+
		\int_{\DB}|f^{(n)}(z)|^2P_\mu(z)(1-|z|^2)^{n-1}\diff A(z)
		\\		
		\lesssim {} &
		\DC_{\mu,n}(f')+\DC_{\mu,n}(f).
		\end{align*}
		The second term is controlled by use of Propositions \ref{prop:RelationLambdaAndMu} and \ref{prop:GeneralizedLittlewood--PaleyTheorem}:
		\begin{align*}
		\DC_{\mu,n}(f)\lesssim |\hat f(n)|^2+\DC_{\lambda,n+1}(f)\approx \DC_{\lambda,n-1}(f')\lesssim |\hat f(n)|^2+\DC_{\mu,n}(f').
		\end{align*}
	\end{proof}
	Letting $f(z)=z^{n-1}$, it is clear that the term $|\hat f(n-1)|^2$ may be excluded in the above proposition, only if $\mu=0$ or $n=0$. In order to obtain boundedness of $M_z$, we need to control $|\hat f(n-1)|^2$ by some lower order Dirichlet integral. A natural way of doing this is provided by the next lemma:
	
	\begin{lemma}\label{lemma:ControlOfFourierCoefficients}
		Let $\mu$ be a non-vanishing finite positive measure on $\TB$. The following are equivalent:
		\begin{enumerate}[(i)]
			\item $|\hat f(0)|^2\lesssim \DC_{\mu,0}(f)$ for $f\in\DC_a$.
			\item Given $n\in\ZB_{\ge 0}$, it holds that $|\hat f(n)|^2\lesssim \DC_{\mu,0}(f)$ for $f\in\DC_a$.
			\item The constant function $1$ is not in the $L^2(\TB,\diff\mu)$-closure of $M_z\DC_a$.
			\item If $\diff \mu=h\diff\lambda+\diff\mu_s$ is the Lebesgue decomposition of $\mu$, then $\log h\in L^1(\TB,\diff\lambda)$.
		\end{enumerate}
	\end{lemma}
	\begin{proof}
		Most of this is covered in \cite[Chapter 4]{Hoffman1962:BanachSpacesOfAnalyticFunctions}. For the equivalence of $(iii)$ and $(iv)$, see the discussion after the Szegö; Kolmogoroff--Krein theorem. The equivalence of $(i)$ and $(iii)$ is Exercise 4. Moreover, $(ii)$ trivially implies $(i)$. Hence, we only need to prove the converse of this.
		
		By hypothesis, there exists an $n_0\in\ZB_{\ge 0}$ such that the desired conclusion holds for all non-negative integers $n\le n_0$. We prove that whenever such an $n_0$ exists, then the conclusion also holds when we let $n=n_0+1$.
		
		Given $f\in\DC_a$, let $f_n(z)=\sum_{k=0}^\infty\hat f(k+n)z^k$. Then $\hat f(n)=\hat f_n(0)$. Since $\DC_{\mu,0}(zf)=\DC_{\mu,0}(f)$, our induction hypothesis yields that
		\begin{multline*}
		|\hat f(n)|^2
		\lesssim 
		\DC_{\mu,0}(f_n)
		=
		\DC_{\mu,0}\left(\frac{f-\sum_{k=0}^{n-1}\hat f (k)z^k}{z^n}\right)
		\\
		=
		\DC_{\mu,0}\left(f-\sum_{k=0}^{n-1}\hat f (k)z^k\right)
		\lesssim 
		\DC_{\mu,0}(f)+\sum_{k=0}^{n-1}|\hat f (k)|^2
		\lesssim
		\DC_{\mu,0}(f).
		\end{multline*}
	\end{proof}	
	\begin{remark}
		If $P_{\mu}\ge c>0$ on $\DB$, then we trivially obtain the stronger conclusion that $\|f\|_{H^2}^2\lesssim \DC_{ \mu,0}(f)$ for $f\in\DC_a$. The condition $P_{\mu}\ge c>0$ on $\DB$ is equivalent to that $h\ge c>0$ on $\TB$, e.g. \cite[Chapter I, Theorem 5.3]{Garnett2007:BoundedAnalyticFunctions}.
	\end{remark}
	
	\begin{lemma}\label{lemma:FourierGrowthEstimate}
		Let $n,N\in\ZB_{\ge 0}$. For $\mu_n\in\DC'$, with $|\hat \mu_n(k)|\lesssim\left(1+|k|\right)^N$, it holds that 
		\[
		|\DC_{\mu_n,n}(f)|\lesssim \DC_{\lambda,n}(f)+\DC_{\lambda,2(n\vee N)+2}(f).
		\]
	\end{lemma}
	\begin{proof}
		Using Lemma \ref{lemma:DirichletIntegralAsADoubleSum},
		\begin{multline*}
		|\DC_{\mu_n,n}(f)|
		\lesssim
		\sum_{k,l=n}^\infty \binom{k\wedge l}{n}|\hat f(k)\conj{\hat f(l)}|\left(1+|l-k|\right)^N
		\\
		=
		\DC_{ \lambda,n}(f)+\sum_{\substack{k,l=n\\k\ne l}}^\infty \binom{k\wedge l}{n}|\hat f(k)\conj{\hat f(l)}|\left(1+|l-k|\right)^N.
		\end{multline*}
		It remains to approximate the above sum.
		\begin{multline*}
		\sum_{\substack{k,l=n\\k\ne l}}^\infty \binom{k\wedge l}{n}|\hat f(k)\conj{\hat f(l)}|\left(1+|l-k|\right)^N
		\approx
		\sum_{\substack{k,l=n\\k\ne l}}^\infty \binom{k\wedge l}{n}|\hat f(k)\conj{\hat f(l)}||l-k|^N
		\\
		\lesssim
		\sum_{k,l=n}^\infty (k\wedge l)^{n}(k\vee l)^{N}|\hat f(k)\conj{\hat f(l)}|
		\lesssim
		\sum_{k,l=n}^\infty (kl)^{n\vee N}|\hat f(k)\conj{\hat f(l)}|.
		\end{multline*}
		We now use Cauchy--Schwarz's inequality:
		\begin{multline*}
		\sum_{k,l=n}^\infty (kl)^{n\vee N}|\hat f(k)\conj{\hat f(l)}|
		=
		\left(\sum_{k=n}^\infty k^{n\vee N+1-1}|\hat f(k)|\right)^2
		\\
		\lesssim
		\sum_{k=n}^\infty k^{2(n\vee N)+2}|\hat f(k)|^2
		\approx
		\DC_{ \lambda,n}(f)+\DC_{ \lambda,2(n\vee N)+2}(f).
		\end{multline*}
	\end{proof}
	
	The following result extends Proposition \ref{prop:RelationLambdaAndMu} $(ii)$:
	
	\begin{lemma}\label{lemma:DerivativesOfFiniteMeasures}
		Let $\mu$ be a finite measure on $\TB$. For $n\in\ZB_{\ge 1}$, $j\in\ZB_{\ge 0}$, it then holds that
		\[
		|\DC_{D^j\mu,n+j}(f)|\lesssim \DC_{ \lambda,n+j}(f)+\DC_{ \lambda,n+2j+1}(f).
		\]
	\end{lemma}
	\begin{proof}
		Note that $\DC_{\mu,n}(f)$ is additive in $\mu$. By the Jordan decomposition $\mu=\mu^+-\mu^-$, where $\mu^+,\mu^-\ge 0$, it suffices to consider the case where $\mu\ge 0$.
		
		We argue by induction over $j$. The case $j=0$ is covered by Proposition \ref{prop:RelationLambdaAndMu} $(ii)$. Assume now that the statement holds for some $j=j_0$. Proposition \ref{prop:DifferentiationPlaysWithDirichletIntegral} now implies that
		\[
		|\DC_{D^{j_0+1}\mu,n+j_0+1}(f)|\lesssim |\DC_{D^{j_0}\mu,n+2+j_0}(f)|+|\DC_{D^{j_0}\mu ,n+1+j_0}(f)|+|\DC_{D^{j_0}\mu,n+j_0}(f')|.
		\]
		By the assumption for $j=j_0$, it holds that
		\begin{align*}
		|\DC_{D^{j_0}\mu,(n+2)+j_0}(f)|
		\lesssim 
		\DC_{ \lambda,n+2+j_0}(f)+\DC_{ \lambda,n+2+2j_0+1}(f),
		\end{align*}
		\begin{align*}
		|\DC_{D^{j_0}\mu,(n+1)+j_0}(f)|
		\lesssim 
		\DC_{ \lambda,n+1+j_0}(f)+\DC_{ \lambda,n+2+2j_0+1}(f),
		\end{align*}
		and
		\begin{multline*}
		|\DC_{D^{j_0}\mu,n+j_0}(f')|
		\lesssim 
		\DC_{ \lambda,n+j_0}(f')+\DC_{ \lambda,n+2j_0+1}(f')
		\\
		\approx
		\DC_{ \lambda,n+j_0+1}(f)
		+
		\DC_{ \lambda,n+j_0+2}(f)
		+
		\DC_{ \lambda,n+2j_0+2}(f)
		+
		\DC_{ \lambda,n+2j_0+3}(f).
		\end{multline*}
		In the last step we have used Proposition \ref{prop:GeneralizedLittlewood--PaleyTheorem}. 
		
		As we already noted, prior to Proposition \ref{prop:GeneralizedLittlewood--PaleyTheorem}, $\DC_{ \lambda,n}(f)=\sum_{k=0}^{\infty}\binom{k}{n}|\hat f(k)|^2$. From this, it is clear that $\DC_{ \lambda,n+j_0+1}(f)$ and $\DC_{ \lambda,n+2j_0+3}(f)$ together dominate all of the above Dirichlet integrals. Piecing together the above estimates therefore yields
		\[
		|\DC_{D^{j_0+1}\mu,n+j_0+1}(f)|\lesssim \DC_{ \lambda,n+j_0+1}(f)+\DC_{ \lambda,n+2(j_0+1)+1}(f).
		\]
	\end{proof}

	\section{Allowable $m$-tuples}\label{sec:AllowableMTuples}
	
	We have the following necessary conditions on allowable $m$-tuples:
	\begin{proposition}\label{prop:PropertiesOfmuk}
		Let $T$ be a bounded $m$-isometry with a cyclic unit vector $e$, and $\MU=(\mu_0,\ldots,\mu_{m-1})$ the corresponding allowable $m$-tuple. 
		\begin{enumerate}[(i)]
			\item For $0\le n\le m-2$, it holds that $|\mu_n(k)|\lesssim\left(1+|k|\right)^\frac{m-1}{2}$. 
			\item For $n=0$, we have the additional property that $\hat \mu_0(0)\ge 0$, with equality if and only if every $\mu_n=0$.
			\item For $n=m-1$, it holds that $\mu_{m-1}\ge 0$, and
			\begin{equation}\label{eq:IntegralRepOfDefectOperator}
			\langle \beta_{m-1}(T)f(T)e,f(T)e\rangle = \int_{\TB} |f|^2 \diff\mu_{m-1},
			\end{equation}
			whenever $f\in\DC_a$.
		\end{enumerate}
	\end{proposition}
	\begin{proof}
		In order to prove the first statement, it suffices to consider $k\ge 0$. Since $\MU$ is normalized, $e$ is a unit vector. Cauchy--Schwarz's inequality yields
		\[
		|\hat \mu_n(k)|\le \|\beta_n(T)\|_{\LC}\|T^ke\|_\HC.
		\]
		By Proposition \ref{prop:PropertiesOfBetak} $(ii)$, $\|T^ke\|_\HC^2$ is a polynomial in $k$, its degree at most $m-1$. The first statement follows. The second statement is the observation that $\mu_0(1)=\hat \mu_0(0)=\|1\|_{\MU}\ge 0$. Since $1$ is cyclic for $M_{\MU}$, we have equality if and only if $\DC_{\MU}^2=\{0\}$. The third statement is just Proposition \ref{prop:PropertiesOfBetak} $(iii)$, with $x=e$ and $\mu_x=\mu_{m-1}$.
	\end{proof}
	
	If $T$ is (say) norm-expanding, i.e. $\beta_1(T)\ge 0$, then $(i)$ in the above result can be improved:
	
	\begin{proposition}\label{prop:PropertiesOfmukImproved}
		Let $T$ be a bounded $m$-isometry with a cyclic unit vector $e$, and $\MU=(\mu_0,\ldots,\mu_{m-1})$ the corresponding allowable $m$-tuple. If $\beta_{n}(T)\ge 0$, then $|\mu_{n+j}(k)|\lesssim\left(1+|k|\right)^\frac{m-1-n}{2}$ for $j\in\ZB_{\ge 0}$.
	\end{proposition}
	\begin{proof}
		By Proposition \ref{prop:PropertiesOfBetak} $(i)$, 
		\begin{multline*}
		|\hat \mu_{n+j}(k)|
		\lesssim 
		\sum_{i=0}^j|\langle T^{*i}\beta_{n}(T)T^ie,T^ke\rangle|
		\\
		=
		\sum_{i=0}^j|\langle \beta_{n}(T)^{1/2}T^ie,\beta_{n}(T)^{1/2}T^{k+i}e\rangle|
		\lesssim
		\sum_{i=0}^j\|\beta_{n}(T)^{1/2}T^{k+i}e\|.
		\end{multline*}
		By Proposition \ref{prop:PropertiesOfBetak} $(ii)$, each $\|\beta_{n}(T)^{1/2}T^{k+i}e\|^2$ is a polynomial in $k$, its degree at most $m-1-n$.
	\end{proof}
	
	If we let $(e_n)_{n=0}^{m-1}$ denote the canonical basis for $\RB^{m}$, then it is natural to interpret the formal product $\mu e_n$ as the $m$-tuple $(\mu_{0},\ldots,\mu_{m-1})$ with $\mu_{n'}=\mu\delta_{n',n}$. We now present some sufficient conditions for a tuple to be allowable. The first one is based on some rather coarse estimates, but still demonstrates the richness of the set of allowable tuples. The main idea behind the proof is contained in the following lemma. The reason for introducing the auxiliary distribution $\nu$ is explained in Remark \ref{remark:ExplanationForAuxiliaryDistribution}:
	
	\begin{lemma}\label{lemma:IdeaOfAllowabilityProofs}
		Let $\MU\in\left(\DC'\right)^m$, and $\nu\in\DC'$. Assume that
		\begin{enumerate}[(i)]
			\item $\nu\ge 0$ and $|\hat f(0)|^2\lesssim \DC_{\nu,0}(f)$ for $f\in\DC_a$.
			\item $\mu_{m-1}$ is positive and non-vanishing.
			\item $\left|\sum_{n=0}^{m-1}\DC_{ \mu_n,n}(f)\right|\lesssim \sum_{k=0}^{m-2}|\hat f(k)|^2+\DC_{ \mu_{m-1},m-1}(f)$.
		\end{enumerate}
		Then $\MU_{\nu,C}=C\nu e_0+\MU+C\mu_{m-1}e_{m-1}$ is an allowable tuple, provided that $C>0$ is sufficiently large.
	\end{lemma}
	
	\begin{proof}
		Combining $(i)$ and $(iii)$ with Lemma \ref{lemma:ControlOfFourierCoefficients}, there exists a $C>0$ such that 
		\[
		\left|\sum_{n=0}^{m-1}\DC_{ \mu_n,n}(f)\right|\le \frac{C}{2} \left(\DC_{ \nu,0}(f)+\DC_{ \mu_{m-1},m-1}(f)\right).
		\]
		Hence,
		\begin{equation*}
		\|f\|_{\MU_{\nu,C}}^2
		=
		C\DC_{ \nu,0}(f)+\sum_{n=0}^{m-1}\DC_{ \mu_n,n}(f)+C\DC_{ \mu_{m-1},m-1}(f)
		\approx \DC_{ \nu,0}(f)+\DC_{ \mu_{m-1},m-1}(f).
		\end{equation*}
		In particular, the quadratic form $f\mapsto \|f\|_{\MU_{\nu,C}}^2$ is positive definite. Moreover,
		\[
		\|M_zf\|_{\MU_{\nu,C}}^2\approx \DC_{ \nu,0}(zf)+\DC_{ \mu_{m-1},m-1}(zf)
		\lesssim \DC_{ \nu,0}(f)+\DC_{ \mu_{m-1},m-1}(f)
		\approx \|f\|_{\MU_{\nu,C}}^2.
		\]
		The above estimate follows by combining Proposition \ref{prop:MzBoundedWRTDirichletIntegrals} with Lemma \ref{lemma:ControlOfFourierCoefficients}.
	\end{proof}
	
	\begin{theorem}\label{thm:SufficientConditionDegenerate}
		Let $m\ge 4$, $\MU\in\left(\DC'\right)^m$, and $\nu\in\DC'$. Assume that
		\begin{enumerate}[(i)]
			\item $\nu\ge 0$ and $|\hat f(0)|^2\lesssim \DC_{\nu,0}(f)$ for $f\in\DC_a$.
			\item if $0\le n\le \frac{m-4}{2}$, then $|\hat \mu_n(k)|\lesssim \left(1+|k|\right)^{\frac{m-4}{2}}$ for $k\in\ZB$.
			\item if $m$ is odd, and $n=\frac{m-3}{2}$, then $\mu_n=\sum_{j=0}^{\frac{m-3}{2}}D^j\nu_j$, where each $\nu_j$ is a finite measure.
			\item if $\frac{m-2}{2}\le n\le m-3$, then $\mu_n=\sum_{j=0}^{m-3-n}D^j\nu_j$, where each $\nu_j$ is a finite measure. In particular $\mu_{m-2}=0$.
			\item $\mu_{m-1}$ is positive and non-vanishing.
		\end{enumerate}
		Then $\MU_{\nu,C}=C\nu e_0+\MU+C\mu_{m-1}e_{m-1}$ is an allowable tuple, provided that $C>0$ is sufficiently large.
	\end{theorem}
	
	\begin{proof}
		We prove that for each $n\in[0,m-2]$, it holds that 
		\begin{equation}\label{eq:InequalityProofForAllowability}
		|\DC_{\mu_n,n}(f)|\lesssim \sum_{k=0}^{m-2}|\hat f(k)|^2+\DC_{ \mu_{m-1},m-1}(f) \quad \textnormal{for}\quad  f\in\DC_a.
		\end{equation}
		By Lemma \ref{lemma:IdeaOfAllowabilityProofs}, this implies that $\MU_{\nu,C}$ is allowable. We prove \eqref{eq:InequalityProofForAllowability} by showing that
		\[
		|\DC_{\mu_n,n}(f)|\lesssim \sum_{k=0}^{m-2}|\hat f(k)|^2+\DC_{\lambda,m-2}(f) \quad \textnormal{for}\quad  f\in\DC_a,
		\]
		and appealing to Proposition \ref{prop:RelationLambdaAndMu}.
		
		For $0\le n\le \frac{m-4}{2}$, Lemma \ref{lemma:FourierGrowthEstimate} implies that 
		\[
		|\DC_{\mu_n,n}(f)|\lesssim \DC_{\lambda,n}(f)+\DC_{\lambda,2(n\vee \frac{m-4}{2})+2}(f)=\DC_{ \lambda,n}(f)+\DC_{ \lambda,m-2}(f).
		\]
		For $\frac{m-2}{2}\le n\le m-3$, we have that 
		\[
		|\DC_{\mu_n,n}(f)|\lesssim \sum_{j=0}^{m-3-n}|\DC_{D^j\nu_j,n}(f)|.
		\]
		Note that $n-j\ge 1$, and $n+j+1\le m-2$. Lemma \ref{lemma:DerivativesOfFiniteMeasures} now implies that 
		\begin{equation*}
		|\DC_{D^j\nu_j,n}(f)|
		=
		|\DC_{D^j\nu_j,n-j+j}(f)|
		\lesssim
		\DC_{ \lambda,n}(f)+\DC_{ \lambda,n+j+1}(f)
		\lesssim
		\DC_{ \lambda,n}(f)+\DC_{ \lambda,m-2}(f).
		\end{equation*}
		The case $n=\frac{m-3}{2}$ is treated similarly.		
	\end{proof}
	
	We similarly obtain the following theorem.
	
	\begin{theorem}\label{thm:SufficientConditionNon-Degenerate}
		Let $m\ge 3$, $\MU\in\left(\DC'\right)^m$, and $\nu\in \DC'$. Assume that
		\begin{enumerate}[(i)]
			\item $\nu\ge 0$ and $|\hat f(0)|^2\lesssim \DC_{\nu,0}(f)$ for $f\in\DC_a$.
			\item if $0\le n\le \frac{m-3}{2}$, then $|\hat \mu_n(k)|\lesssim \left(1+|k|\right)^{\frac{m-3}{2}}$ for $k\in\ZB$.
			\item if $m$ is even, and $n=\frac{m-2}{2}$, then $\mu_n=\sum_{j=0}^{\frac{m-4}{2}}D^j\nu_j$, where each $\nu_j$ is a finite measure.
			\item if $\frac{m-1}{2}\le n\le m-2$, then $\mu_n=\sum_{j=0}^{m-2-n}D^j\nu_j$, where each $\nu_j$ is a finite measure.
			\item $P_{\mu_{m-1}}\ge c>0$ on $\DB$.
		\end{enumerate}
		Then $\MU_{\nu,C}=C\nu e_0+\MU+C\mu_{m-1}e_{m-1}$ is an allowable tuple, provided that $C>0$ is sufficiently large.
	\end{theorem}
	
	\begin{proof}
		The hypothesis on $\mu_{m-1}$ implies that $\DC_{\lambda,m-1}(f)\lesssim \DC_{\mu_{m-1},m-1}(f)$. Now, \eqref{eq:InequalityProofForAllowability} follows by showing that 
		\[
		|\DC_{\mu_n,n}(f)|\lesssim \sum_{k=0}^{m-2}|\hat f(k)|^2+\DC_{\lambda,m-1}(f) \quad \textnormal{for}\quad  f\in\DC_a,
		\]
		Adapting the proof of Theorem \ref{thm:SufficientConditionDegenerate} is now a matter of bookkeeping.
	\end{proof}
	
	\begin{remark}\label{remark:ExplanationForAuxiliaryDistribution}
		The simplest instance of the two theorems above is when $\mu_0=\nu$. The chosen formulation is an attempt to state that for $\MU$ to be allowable, it is not necessary for $\mu_0$ to be a positive measure. It is sufficient for it to ``contain a sufficiently large positive part''.
	\end{remark}
	
	\begin{example}\label{ex:NegativeMeasure}
		Let $\MU$ be a $3$-tuple, where $P_{\mu_0},P_{\mu_2}\ge c>0$ on $\DB$, and $\mu_{1}$ is a negative measure on $\TB$. If the total variation of $\mu_1$ is sufficiently small, then $\MU$ is allowable, according to Theorem \ref{thm:SufficientConditionNon-Degenerate}.
	\end{example}
	
	Despite the richness of the preceding theorems, they do not cover the following simple situation:
	
	\begin{theorem}\label{thm:SufficientPositiveMeasures}
		Let $\MU=(\mu_{0},\ldots,\mu_{m-1})$ be an $m$-tuple of finite positive measures, where $\mu_{m-1}$ is non-vanishing.
		\begin{enumerate}[(i)]
			\item If $m=1$, then $\MU$ is allowable.
			\item If $m\ge 2$, and $|\hat f(0)|^2\lesssim \DC_{\mu_0,0}(f)$ for $f\in\DC_a$, then $\MU$ is allowable.
		\end{enumerate}
	\end{theorem}
	
	\begin{proof}
		It is clear that $f\mapsto\|f\|_{\MU}^2$ is positive definite. By Proposition \ref{prop:MzBoundedWRTDirichletIntegrals},
		\[
		\|M_z\|_{\MU}^2\lesssim \|f\|_{\MU}^2+\sum_{k=0}^{m-2}|\hat f(k)|^2.
		\]
		For $m=1$, the above sum is empty, and $M_z$ is isometric. For $m\ge 2$, we need to control the terms $|\hat f(k)|^2$. This is done using Lemma \ref{lemma:ControlOfFourierCoefficients}.
	\end{proof}
	
	\begin{remark}\label{remark:NormExpansion}
		By Proposition \ref{prop:MzPlaysNicelyWithDirichletIntegrals}, if $\MU$ is an allowable tuple of positive measures, then $M_{\MU}$ is norm-expanding, i.e. $\|f\|_{\MU}^2\le \|M_{\MU}f\|_{\MU}^2$. On the other hand, if $\MU$ is an allowable tuple, where $\mu_1$ is a non-vanishing negative measure, cf. Example \ref{ex:NegativeMeasure}, then $\|M_{\MU}1\|< \|1\|_{\MU}$, i.e. $M_{\MU}$ is \textit{not} norm-expanding.
	\end{remark}
	
	It seems noteworthy that if $\MU\in\left(\DC'\right)^m$ satisfies the hypothesis of Theorem \ref{thm:SufficientConditionDegenerate}, except that $\mu_{m-2}\ge 0$ is non-vanishing, then $\MU_{\nu,C}$ is still an allowable tuple. To see this, combine Theorems \ref{thm:SufficientConditionDegenerate} and \ref{thm:SufficientPositiveMeasures} with the following observation: 
	
	\begin{proposition}
		If $\MU_1$ and $\MU_2$ are allowable $m$-tuples, then $\MU_1+\MU_2$ is also an allowable $m$-tuple.
	\end{proposition}
	\begin{remark}
		Even though we do not provide an example of this, one should not exclude the possibility that the quadratic forms $f\mapsto \|f\|_{\MU_1}^2$ and $f\mapsto \|f\|_{\MU_2}^2$ have different kernels.
	\end{remark}
	
	The following example bears several insights:
	\begin{example}\label{ex:OuterTermsDoNotControlTheMiddle}
		Let $\mu_0=\lambda$, $\mu_1=\delta_1$, and $\mu_2=\delta_{-1}$. By Theorem \ref{thm:SufficientPositiveMeasures}, $(\mu_0,\mu_1,\mu_2)$ is an allowable $3$-tuple. Note that $\hat \mu_1(k)=1$, while $\hat \mu_{2}(k)=(-1)^k$.
		
		Let $f(z)=\log \frac{1}{1-rz} = \sum_{k=1}^\infty \frac{r^k}{k} z^{k}$. In the proof of Proposition \ref{prop:PmudAIsCarlesonMeasure}, we essentially showed that
		\begin{align*}
		\DC_{\mu_1,1}(f)
		&=
		\log\left(\frac{1}{1-r^2}\right)
		+
		\frac{2r}{1-r}\log\left(\frac{1}{1-r}\right)
		-
		\frac{2}{1-r}\log\left(\frac{1}{1-r^2}\right)
		\\
		&=
		\frac{1+r}{1-r}\log(1+r)+\log(1-r).
		\end{align*}
		A similar (but rather lengthy) calculation, using the identity
		\[
		\sum_{l=2}^{k}(l-1)(-r)^{l-1}=\frac{1}{(1+r)^2}\left(1-k(-r)^{k-1}+(k-1)(-r)^k\right),
		\]
		yields that
		\begin{align*}
		\DC_{\mu_2,2}(f)
		= 
		\frac{(r^2-2r-1)\log(1-r)-r^2}{2(1+r)^2}-\frac{\log(1+r)}{2}.
		\end{align*}
		This can of course also be computed using the standard techniques from any introductory course to calculus, but I doubt that this will be quicker.
		
		By inspection,
		\[
		\lim_{r\to 1^{-}}\frac{\DC_{\mu_1,1}(f)}{\DC_{\mu_2,2}(f)}=\infty.
		\]		
		From this we conclude that $\|f\|_{\MU}^2\not\lesssim\DC_{ \mu_0,0}(f)+\DC_{ \mu_2,2}(f)$. This is in stark contrast to the theory of Sobolev spaces, where similar estimates are standard, e.g. \cite[Chapter 1]{Mazya2011:SobolevSpacesWithApplicationsToEllipticPartialDifferentialEquations}.
		
		The second conclusion is that if $c<0$, then the tuple $(\lambda,c\delta_{1},\delta_{-1})$ fails to be allowable, regardless the magnitude of $c$. This shows that condition $(v)$ in Theorem \ref{thm:SufficientConditionNon-Degenerate} should not be carelessly disregarded. Another way of phrasing this is that in order to prove Theorem \ref{thm:SufficientPositiveMeasures}, one needs something more than the techniques used in the proof of Theorem \ref{thm:SufficientConditionDegenerate}.
	\end{example}
	
	\begin{example}\label{ex:DirichletIntegralIsConditionallyConvergent}
		Let $N\in\ZB_{\ge 0}$, and define $\mu\in\DC'$ by 
		\[
		\hat \mu(k)=
		\left\{
		\begin{array}{cl}
		\binom{k+N}{N}&\textnormal{for }k\ge N,
		\\
		0&\textnormal{otherwise.}
		\end{array}
		\right.
		\]
		Then $P_{\mu}(z)=\frac{1}{(1-z)^{N+1}}$. Define $\mu_R,\mu_I\in\DC'$ by $P_{\mu_R}=\Re P_\mu$ and $P_{\mu_I}=\Im P_\mu$. Then $|\hat \mu_R(k)|,|\hat \mu_I(k)|\lesssim\left(1+|k|\right)^N$. By Theorem \ref{thm:SufficientConditionDegenerate}, given $n$ there exists allowable tuples for which $\mu_n=\mu_R$ and $\mu_{n}=\mu_I$ respectively. 
		
		We claim that if $n\le N$, then at least one of the functions $z\mapsto P_{\mu_R}(z)(1-|z|^2)^{n-1}$ and $z\mapsto P_{\mu_I}(z)(1-|z|^2)^{n-1}$ is not in $L^1(\DB,\diff A)$, hence the corresponding Dirichlet integral $\DC_{ \mu_n,n}(f)$ is conditionally convergent, at least for some $f\in\DC_a$. Indeed, $|P_{\mu_R}|+|P_{\mu_I}|\ge |P_{\mu}|$, and
		\[
		\int_{\DB}|P_{\mu}(z)|(1-|z|^2)^{n-1}\ge \int_{S}\frac{1}{|1-z|^{N+1}}(1-|z|^2)^{n-1}\diff A(z),
		\]
		where $S$ is the set $\{z=1+\rho e^{i\theta};0<\rho< \frac{1}{2}, \frac{2\pi}{3}<\theta<\frac{4\pi}{3}\}$. On this sector, $1-|z|^2>\frac{\rho}{2}$. By integration in polar coordinates,
		\[
		\int_{S}\frac{1}{|1-z|^{N+1}}(1-|z|^2)^{n-1}\diff A(z)
		\gtrsim 
		\int_{\rho=0}^{1/2}\rho^{n-N-1}\diff\rho=\infty,
		\]
		provided that $n\le N$.
	\end{example}

	\section{Concluding remarks}\label{sec:ConcludingRemarks}
	
	Let us for a moment consider the possibility of unbounded $m$-isometries. The machinery developed for proving the first part of Theorem \ref{thm:ModelTheorem} effectively constructs quadratic forms $\|\cdot\|_{\MU}^2$ on $\DC_a$ with respect to which $M_z$ is an $m$-isometry, i.e.
	\[
	\sum_{j=0}^{m}\left(-1\right)^{m-j}\binom{m}{j}\|M_{\MU}^{m-j}f\|_{\MU}^2=0\quad \textnormal{for}\quad f\in\DC_a.
	\]
	We then impose the conditions that $f\mapsto\|f\|_{\MU}^2$ is positive semi-definite, to ensure that $\DC_{\MU}^2$ becomes a Hilbert space, and that $M_{\MU}$ acts like a bounded operator on this space, but these assumptions have nothing to do with the above equation. With this in mind, it is interesting to see what can happen if we disregard hypothesis $(ii)$ in Theorem \ref{thm:SufficientPositiveMeasures}: 
	
	\begin{example}\label{ex:Unbounded2Isometry}
		Consider the tuple $\MU=(\delta_{-1},\delta_1)$. The form $\|\cdot\|_{\MU}^2$ is positive definite, so $M_{\MU}$ is a densely defined operator on the Hilbert space $\DC_{\MU}^2$. Moreover, $M_{\MU}$ is a $2$-isometry.
		
		Let now $g_r(z)=\frac{(1-r^2)^{1/2}}{1+rz}$, and $f_r(z)=2g_r(-1)+(z-1)g_r(z)$. Then $f_r(-1)=0$, and $f_r(1)=2g_r(-1)\to\infty$ as $r\to 1^-$. A much deeper statement is that $\DC_{\delta_1,1}(f_r)=\|g_r\|_{H^2}^2=1$, e.g. \cite[Theorem 7.2.1]{ElFallah-Kellay-Mashreghi-Ransford2014:APrimerOnTheDirichletSpace}. This implies that $\|f_r\|_{\MU}^2=1$, whereas, by Proposition \ref{prop:MzPlaysNicelyWithDirichletIntegrals}, 
		\[
		\|zf_r\|_{\MU}^2=\|f_r\|_{\MU}^2+\DC_{\delta_{1},0}(f_r)=1+f_r(1)\to\infty\quad \textnormal{as}\quad r\to 1^-.
		\]
		Hence $M_{\MU}$ is unbounded.
	\end{example}
	
	It is already known that unbounded $2$-isometries exist, see \cite[Example 3.4]{Bermudez-Martinon-Muller2014:mq-IsometriesOnMetricSpaces}. However, the above example seems relatively simple. 
	
	We now return to the convention that operators referred to as $m$-isometries are assumed to be bounded.
	
	Let $\DC_{\lambda,m}^2$ denote the space of analytic functions $f:\DB\to\CB$ such that $\DC_{\lambda,0}(f)+\DC_{\lambda,m}(f)<\infty$. This is a subspace of $H^2$. We have seen in Example \ref{ex:OuterTermsDoNotControlTheMiddle} that for some allowable tuples $\MU$, $\|f\|_{\MU}^2\not\lesssim \DC_{\lambda,0}(f)+\DC_{\mu_{m-1},m-1}(f)$. On the other hand, we have not seen a counter example to the estimate $\DC_{\lambda,0}(f)+\DC_{\mu_{m-1},m-1}(f)\lesssim \|f\|_{\MU}^2$. 
	\begin{question}\label{q:DoWeHaveGeneralEmbedding?}
		Let $m\in\ZB_{\ge 2}$, and suppose that $\MU$ is an allowable $m$-tuple, with $\mu_{m-1}\ne 0$. Is it then true that $\DC_{\MU}^2\hookrightarrow\DC_{ \lambda,m-2}^2$? 
	\end{question}
	This a priori bound on the size of the model space is well-known in the case where $m=2$ and $e\in\ker T^*$, e.g. \cite[Theorem 7.1.2]{ElFallah-Kellay-Mashreghi-Ransford2014:APrimerOnTheDirichletSpace}. Heuristically, an affirmative answer to Question \ref{q:DoWeHaveGeneralEmbedding?} would connect nicely to the theory of $\gamma$-hypercontractions, where small classes of operators correspond to large model spaces, e.g. \cite{Olofsson2015:PartsOfAdjointWeightedShifts}.
	
	Recall that $m$-isometries are bounded from below. Some results related to Question \ref{q:DoWeHaveGeneralEmbedding?} are the following:
	\begin{proposition}\label{prop:GeneralEmbeddingWithWrongDomain}
		If $T\in\LC$ is an $m$-isometry with a cyclic unit vector $e\in\ker T^*$, then $\|\cdot\|_{T,e}$ is a proper norm. Moreover,
		\[
		\|f(T)e\|_\HC^2\ge \sum_{k=0}^{\infty}c^{2k}|\hat f(k)|^2,
		\]
		whenever $c>0$ satisfies $\|Tx\|_\HC\ge c\|x\|_\HC$. In particular, $\DC_{T,e}^2\hookrightarrow H^2$ whenever $T$ is norm-expanding.
	\end{proposition}
	\begin{proof}
		By Lemma \ref{lemma:SmoothFunctionalCalculusIsometry}, $\|\cdot\|_{T,e}$ is a seminorm. The above inequality implies that it is positive definite, hence a proper norm. It remains to prove the inequality:
		
		Given $f\in\DC_a$, the orthogonal decomposition 
		\[
		f(T)e=\hat f(0)e\oplus\sum_{k=1}^{\infty}\hat f(k)T^ke\in\ker T^*\oplus T\HC
		\]
		implies that 
		\[
		\|f(T)e\|_\HC^2=|\hat f(0)|^2+\|\sum_{k=1}^{\infty}\hat f(k)T^ke\|_\HC^2\ge |\hat f(0)|^2+c^2\|\sum_{k=0}^{\infty}\hat f(k+1)T^ke\|_\HC^2. 
		\]
		The desired conclusion follows by an induction argument.
	\end{proof}
	\begin{proposition}
		Let $T\in\LC$ be an $m$-isometry with cyclic unit vector $e$, and corresponding allowable tuple $\MU$. If $P_{\mu_{m-1}}\ge c>0$ on $\DB$, then $\DC_{T,e}^2\hookrightarrow H^2$.
	\end{proposition}
	\begin{proof}
		By Proposition \ref{prop:PropertiesOfmuk} and Lemma \ref{lemma:SmoothFunctionalCalculusIsometry}, $\DC_{T,e}^2\hookrightarrow L^2(\TB,\diff \mu_{m-1})$:
		\[
		\int_{\TB}|f|^2\diff \mu_{m-1} 
		= 
		\langle \beta_{m-1}(T)f(T)e,f(T)e\rangle_{\HC}
		\lesssim
		\|f(T)e\|_\HC^2=\|f\|_{T,e}^2.
		\]
		Our hypothesis clearly implies that $L^2(\TB,\diff \mu_{m-1}) \hookrightarrow H^2$. 
	\end{proof}
	
	Let $T$ be an $m$-isometry. It is known that $\ker\beta_{m-1}(T)$ is the largest $T$-invariant subspace on which $T$ acts like an $(m-1)$-isometry, see \cite[Proposition 1.6]{Agler-Stankus1995:m-IsometricTransformationsOfHilbertSpaceI}.
	A heuristic explanation to why the tuple $\MU=(\lambda,c\delta_{1},\delta_{-1})$, where $c<0$, presented in Example \ref{ex:OuterTermsDoNotControlTheMiddle} fails to be allowable is that, if $\mu_2=\delta_{-1}$, then $\mu_1$ should ``almost'' be a positive measure, because $M_{\MU}$ is ``almost'' a $2$-isometry: By Proposition \ref{prop:PropertiesOfmuk}, 
	\[
	\langle \beta_2(M_{\MU})f,f\rangle_{\MU}=\int_{\TB}|f|^2\diff \mu_{2}=|f(-1)|^2.
	\]
	This implies that $\ker \beta_2(M_{\MU})$ is a space of codimension $1$. 
	
	On the other hand, suppose  that $|\hat f(0)|^2\lesssim \int_{\TB}|f|^2\diff \mu_{m-1}$. By Lemma \ref{lemma:ControlOfFourierCoefficients}, $\int_{\TB}|f|^2\diff \mu_{m-1}=0\Rightarrow f\equiv 0$ on $\DB$, so $M{\MU}$ is far from being an $(m-1)$-isometry, in the sense that it has no non-trivial invariant subspace on which it acts like an $(m-1)$-isometry.
	
	Note that $\int_{\TB}|f|^2\diff \mu_{m-1}=\DC_{ \mu_{m-1},0}(f)$. Hence, a condition imposed on $\mu_0$, in order to make $M_z$ bounded (Theorems \ref{thm:SufficientConditionDegenerate}, \ref{thm:SufficientConditionNon-Degenerate}, and \ref{thm:SufficientPositiveMeasures}), may also be imposed on $\mu_{m-1}$, in order to make $M_{\MU}$ far from being $(m-1)$-isometric. This presents two ways in which properties of the (possibly unbounded) Hilbert space operator $M_{\MU}$ are reflected in properties of the tuple $\MU$. It would certainly be interesting to see a further investigation of this connection between operator theory and harmonic analysis.
	
	As was mentioned in the introduction, $m$-isometries having the wandering subspace property have been studied by Shimorin \cite{Shimorin2001:Wold-TypeDecompositionsAndWanderingSubspacesForOperatorsCloseToIsometries}. We conclude this note with a remark related to the hard earned Proposition 3.15 of the cited work: 
	
	Suppose that $T$ is an analytic $3$-isometry with $\dim\ker T^*=1$ and let $e\in\ker T^*$ have unit length. If in addition $Te\perp T^2\HC$, and $T$ is norm-expanding, i.e. $T^*T\ge I$, then a certain family of dilation operators is uniformly bounded. The uniform bound in turn implies that $e$ is cyclic for $T$. 
	
	In our setting, $e\in\ker T^*$ and $\|e\|_\HC^2=1$ if and only if $\mu_0=\lambda$. The condition $Te\perp T^2\HC$ implies that $\mu_1=c\lambda$ for some $c\in\RB$. The condition $T^*T\ge I$ implies that $c\ge 0$. Our results show that these additional conditions are far from necessary for a $3$-isometry to have the wandering subspace property. Indeed, $\mu_1$ does not need to be a multiple of $\lambda$ (Theorems \ref{thm:SufficientConditionDegenerate}, \ref{thm:SufficientConditionNon-Degenerate}, and \ref{thm:SufficientPositiveMeasures}), and $T$ does not need to be norm-expanding, c.f. Remark \ref{remark:NormExpansion}.

\end{document}